\documentclass[11pt,hyp,]{nyjm}
\usepackage{hyperref}
\hypersetup{nesting=true,debug=true,naturalnames=true}
\usepackage{graphicx,amssymb,upref}
\title[Genus theory and $\varepsilon$-conjectures]
{Genus theory and $\varepsilon$-conjectures \\ on $p$-class groups} 
\author{Georges Gras}
\address{Villa la Gardette, 4 Chemin Ch\^ateau Gagni\`ere, 
F--38520, Le Bourg d'Oisans.}
\email{g.mn.gras@wanadoo.fr}


\newtheorem{theorem}{Theorem}[section]
\newtheorem{lemma}[theorem]{Lemma}

\newtheorem{conjecture}[theorem]{Conjecture}
\newtheorem{proposition}[theorem]{Proposition}
\theoremstyle{definition}
\newtheorem{definition}[theorem]{Definition}
\theoremstyle{remark}
\newtheorem{remark}[theorem]{Remark}
\newtheorem{remarks}[theorem]{Remarks}
\newcommand{\order}{\raise1.5pt \hbox{${\scriptscriptstyle \#}$}}
\def\lien{\mathrel{\mkern-4mu}}
\def\too{\relbar\lien\rightarrow}
\def\tooo{\relbar\lien\relbar\lien\too}

\newcommand{\ds}{\displaystyle}

\newcommand{\R}{\mathbb{R}}

\newcommand{\Q}{\mathbb{Q}}
\newcommand{\Z}{\mathbb{Z}}
\newcommand{\F}{\mathbb{F}}
\newcommand{\Cl}{{\mathcal C}\hskip-2pt{\ell}}
\newcommand{\cl}{c\hskip-1pt{\ell}}
\newcommand{\plus}{\ds\mathop{\raise 2.0pt \hbox{$\bigoplus$}}\limits}
\newcommand{\prd}{\ds\mathop{\raise 2.0pt \hbox{$\prod$}}\limits}
\newcommand{\sm}{\ds\mathop{\raise 2.0pt \hbox{$\sum$}}\limits}
\newcommand{\wt}{\widetilde}
\newcommand{\ov}{\overline} 
\newcommand{\wh}{\widehat} 

\begin{document}

\date{June 25, 2019} %

\begin{abstract} We suspect that the ``genus part'' of the class 
number of a number field $K$ may be an obstruction for an ``easy 
proof'' of the classical $p$-rank $\varepsilon$-conjecture for $p$-class 
groups and, a fortiori, for a proof of the ``strong $\varepsilon$-conjecture'': 
$\order (\Cl_K \otimes \Z_p) \ll_{d,p,\varepsilon} 
(\sqrt{D_K}\,)^\varepsilon$ for all $K$ of degree $d$. 
We analyze the weight of genus theory in this inequality by means of an
infinite family of degree $p$ cyclic fields with many ramified primes, then 
we prove the $p$-rank $\varepsilon$-conjecture: 
$\order (\Cl_K \otimes \F_p) \ll_{d,p,\varepsilon} 
(\sqrt{D_K}\,)^\varepsilon$, for $d=p$ and the family of degree $p$ cyclic  
extensions (Theorem \ref{thmf}) then sketch the case of arbitrary 
base fields. The possible obstruction for the strong form, in the degree $p$ 
cyclic case, is the order of magnitude of the set of ``exceptional'' $p$-classes 
given by a well-known non-predictible algorithm, but controled thanks to 
recent density results due to Koymans--Pagano. Then we compare the 
$\varepsilon$-conjectures with some $p$-adic conjectures, 
of Brauer--Siegel type, about the torsion group ${\mathcal T}_K$ 
of the Galois group of the maximal abelian $p$-ramified 
pro-$p$-extension of totally real number fields $K$. We give 
numerical computations with the corresponding PARI/GP programs.
\end{abstract}

\maketitle

\tableofcontents 

\section{Introduction} 
For any number field $K$, we denote by $\Cl_K$ the 
class group of $K$ in the restricted sense and by $\Cl_K \otimes \Z_p$ its 
$p$-class group, for any prime number $p$; to avoid any ambiguity, we shall
write $\Cl_K \otimes \F_p$ for the ``$p$-torsion group'', often denoted
$\Cl_K[p]$ in most papers, only giving the $p$-rank ${\rm rk}_p(\Cl_K)$ of 
$\Cl_K$.
One knows the following classical result (weak form of theorems
of Brauer, Brauer--Siegel, Tsfasman--Vlad\u{u}\c{t}--Zykin \cite{TV,Zy}): 

\medskip
{\it For all $\varepsilon >0$ there exists a constant $C_{d, \varepsilon}$ such that
$\order \Cl_K \leq C_{d, \varepsilon} \cdot (\sqrt{ D_K}\,)^{1+\varepsilon}$,
where $d$ is the degree of $K$ and $D_K$ the absolute value of its discriminant.}

\medskip
But it is clear that, arithmetically, the behaviour of the $p$-Sylow subgroups of 
$\Cl_K$ depends, in conflicting manners (regarding $p$), on many parameters
(signature, ramification, prime divisors of $d$, action of Galois groups, etc.).

\smallskip
After the Cohen--Lenstra--Martinet, Adam--Malle, Delaunay--Jouhet, Gerth,
Koymans--Pagano,... heuristics, conjectures, or density statements,
on the order and structure of $\Cl_K \otimes \Z_ p$ 
\cite{CL,CM,AM,Mal,DJ,Ge,KP}, many authors study and
prove inequalities of the form $\order (\Cl_K \otimes \F_ p) \leq C_{d, p,\varepsilon} 
\cdot (\sqrt{D_K}\,)^{\,c+\varepsilon}$, with positive constant $c <1$ as small 
as possible (e.g., under GRH, the inequality $\order (\Cl_K \otimes \F_p) 
\leq C_{d,p,\varepsilon} \cdot (\sqrt{D_K}\,)^{\,1 - \frac{1}{p\,(d-1)} + \varepsilon}$
\cite[Proposition 1]{EV}; see also \cite[\S\,1.1]{FW} for more examples and 
comments). The various links between this $\varepsilon$-conjecture and the 
above classical heuristics (or results) are described in \cite[\S\,1.1, Theorem 1.2, 
Remark 3.3]{P-TB-W}.

\smallskip
For a general history upon today about such inequalities, 
we refer to some recent papers of the bibliography (e.g., 
\cite{EPW,EV,FW,P-TB-W,W}) in which the reader can have a more 
complete list of recent contributions.

 \smallskip
For short, we shall call ``$p$-rank $\varepsilon$-conjecture'' the case:
$$\order (\Cl_K \otimes \F_p)  \leq 
C_{d, p,\varepsilon}\cdot (\sqrt{D_K}\,)^\varepsilon ,$$

and ``strong $\varepsilon$-conjecture'' the case:
$$\order (\Cl_K \otimes \Z_p) \leq C_{d, p,\varepsilon}
\cdot (\sqrt{D_K}\,)^\varepsilon. $$

These kind of results do not separate the case of totally real fields, for which 
we know that the class number is rather small, regarding the case 
of non totally real fields (see the tables in Washington's book \cite{Wa}
and papers by Schoof as \cite{Sc1,Sc2}, among many contributions; 
however, some real fields may have exceptional large class numbers 
regarding $D_K$ \cite{Da,DaKM}). 
Moreover, by nature, the results that we have quoted deal with upper 
bounds of the $p$-rank ${\rm rk}_p(\Cl_K)$ and precisely we shall 
see that genus theory gives, when it applies, large and maximal $p$-ranks
and possibly unbounded orders, which probably makes harder 
proofs by classical complex analytic way.

\begin{remark}\label{atinfty}
Nevertheless, many arguments and computations are in 
favor of the strong $\varepsilon$-conjecture
$\order (\Cl_K \otimes \Z_ p) \leq C_{d, p,\varepsilon} 
\cdot (\sqrt{D_K}\,)^{\varepsilon}$, except possibly for
subfamilies of density zero, even if such a conjecture is 
``a conjecture at infinity for fixed $p$'' because, as we shall 
see in numerical calculations, the constants $C_{d, p,\varepsilon}$ 
are enormous, especially for $d=p$,
so that, for ``usual fields'', the inequalities are trivial 
up to some huge values of the discriminant. 
So it will be difficult to give convincing computations ``at infinity''.
\end{remark}

In this context, we shall only prove the case of the $p$-rank 
$\varepsilon$-conjecture for degree $p$ cyclic fields $K$ (Theorem \ref{thmf}):

\medskip
{\bf Main Theorem.} {\it Denote by $\Cl_K$ the class group of any number field $K$.
Let $p \geq 2$ be a prime number; then
the $p$-rank $\varepsilon$-conjecture saying that:
$\order (\Cl_K \otimes \F_p) \ll_{d,p,\varepsilon} (\sqrt{D_K}\,)^\varepsilon$ for all $K$
of degree $d$, is true for the subfamily of all cyclic extension $F/\Q$ 
of degree $d=p$.}

\medskip
For these fields, with Galois group $G =: \langle \sigma \rangle$, 
we shall use the exact sequence:
$$1\to (\Cl_K \otimes \Z_ p)^G \too \Cl_K \otimes \Z_ p
 \too (\Cl_K \otimes \Z_ p)^{1-\sigma} \to 1$$
 
and consider the ``non-genus part'' $(\Cl_K \otimes \Z_ p)^{1-\sigma}$ 
(or set of ``exceptional $p$-classes'') as a random object for which 
densities results are known (see Remark \ref{KP}).

\smallskip
The case $d=p=2$ is known from Gauss genera
theory, but its generalization is not obvious since for $p=2$ the 
$p$-rank is canonical (given by Chevalley's formula \eqref{chevalley})
while for $p>2$ it depends on the algorithm which determines the 
complete structure of $\Cl_K \otimes \Z_ p$; indeed, for $d=p=2$ and any class 
$\gamma$ such that $\gamma^2=1$, we can write
$\gamma^{\sigma-1} = \gamma^{\sigma +1-2} =
{\rm N}_{K/\Q}(\gamma) \, \gamma^{-2}=1$, which does not work for $d=p>2$.

\smallskip
To our knowledge, for $d=p>2$, only the case of cyclic cubic fields is
proved \cite[Corollary~1, case (3)]{EV} for the $3$-rank $\varepsilon$-conjecture.

\smallskip
For CM fields, we shall try to put the ``minus part'' of $\Cl_K \otimes \Z_p$
in ``duality'' with the ``plus part'' of the torsion group ${\mathcal T}_K$ 
of the Galois group of the maximal abelian $p$-ramified pro-$p$-extension 
of $K$ for which, on the contrary, we know that there is no complete strong
$\varepsilon$-conjecture because of some explicit families of density zero.

\smallskip
This suggests that, for any given $p$, the strong $\varepsilon$-conjecture 
may be true ``for almost all field'' in a sense to be specified.

\section{A possible obstruction due to exceptional $p$-classes}

We shall illustrate the above comments with a family of fields with
optimal $p$-ranks, and give numerical illustrations with PARI/GP \cite{P} programs.

\subsection{Definition of the family $(F_{N,p})_{N \geq 1}$}
We consider a fixed prime $p \geq 2$ and the sequence of all odd prime numbers 
$\ell_k$, totally split in $\Q(\mu_p^{})$, whence such that $\ell_k \equiv 1 \pmod p$. 
For $p>2$, let $F_{N,p}$ be any cyclic extension of degree $p$, 
of global conductor $f_{N,p} = \prod_{k=1}^N \ell_i$,  for any $N\geq 1$.
For $p=2$ we shall consider $F_{N,2} := \Q(\sqrt{\pm \,\ell_1 \cdots \ell_N})$ 
$=\Q(\sqrt{\pm \,3 \cdot 5 \cdots \ell_N})$ with $\pm\,$ such that $2$ be unramified, 
and consider the restricted sense for class groups, which is more canonical for our
purpose.
We may consider degree $p$ cyclic  extensions ramified at $p$ without any
modification of the forthcoming reasonings, except an useless complexity 
of redaction (this modifies the above conductors up to the constant factor 
$p^2$ or $8$).

\smallskip
The discriminant $D_{N,p}$ of $F_{N,p}$ is the product of the conductors 
associated to each nontrivial character of $G_{N,p} := {\rm Gal}(F_{N,p}/\Q)$, 
thus $D_{N,p} = f_{N,p}^{p - 1}$. There are $(p-1)^{N-1}$ such fields, all
contained in $\Q(\mu_{\ell_1\cdots \ell_N}^{})$.

\smallskip
The Chevalley formula \cite{Ch} gives the number of ambiguous classes 
(i.e., invariant by $G_{N,p}$) which is equal to the genus number for cyclic fields:
\begin{equation}\label{chevalley}
\order\big( \Cl_{F_{N,p}} \otimes \Z_p \big)^{G_{N,p}} = p^{N-1}.
\end{equation}

It is obvious that the group of ambiguous classes is 
$p$-elementary of $p$-rank $N-1 \leq {\rm rk}_p(\Cl_{F_{N,p}})$ 
and generated by the ramified primes ${\mathfrak l}_i \mid \ell_i$,
$1 \leq i \leq N$, with a (non-trivial) relation 
$\prod_{i=1}^N{\mathfrak l}_i^{n_i} = (a)$, $a \in F_{N,p}^\times$,
due to the classical Kummer theory over $\Q(\mu_p)$
giving $F_{N,p}(\mu_p) = \Q(\mu_p)(\sqrt[p] {\alpha})$, 
$\alpha \in \Q(\mu_p)^\times$ depending of canonical Gauss sums.

\subsection{About the exceptional $p$-classes}\label{algo}
We know that $\order (\Cl_{F_{N,p}} \otimes \F_p)$ and 
$\order (\Cl_{F_{N,p}} \otimes \Z_p)$, both depend of a
non-predictible algorithm that we recall below; this shows that the 
so-called Brumer--Silverman--Duke--Zhang--Ellenberg--Venkatesh 
$\varepsilon$-conjecture on the $p$-ranks is not so different 
from the strong $\varepsilon$-conjecture, in a logical point of view, 
except that we shall see that the $p$-rank is an $O(N)$ contrary to 
$\order (\Cl_{F_{N,p}} \otimes \Z_p)$ whose order of magnitude is unknown.

\begin{definition}
Let $\delta(N) \geq 0$ be such that
$\order (\Cl_{F_{N,p}} \otimes \F_p) = p^{N-1 + \delta(N)}$
and let $\Delta(N) \geq \delta(N)$ be such that
$\order (\Cl_{F_{N,p}} \otimes \Z_p) = p^{N-1 + \Delta(N)}$.
\end{definition}

In other words, $p^{\Delta(N)} = \order (\Cl_{F_{N,p}} \otimes \Z_p)^{1-\sigma}$ 
measures what we shall call the set of ``exceptional $p$-classes'' (i.e., non-invariant)
obtained via the classical filtration of $\Cl_{F_{N,p}} \otimes \Z_p$ (a general 
theoretical and numerical approach was given in \cite{Gr4,Gr5} and \cite{St}
after the historical papers of Inaba \cite{In}, Redei--Reichardt \cite{RR}, 
Fr\"ohlich \cite{Fr1,Fr2} and others; for a wide generalization to 
ray class groups and a survey, see \cite{Gr3}). Then $N-1 + \delta(N)$ is the $p$-rank. 

\begin{remark}\label{KP}
During the writing of this paper we have been informed, by Peter Koymans 
and Carlo Pagano, of their work \cite{KP} proving that the $p$-class
groups of cyclic degree $p$ fields have, under GRH, the (explicit) distribution 
conjectured by Frank Gerth III \cite{Ge} and generalizing results of Jack Klys by
means of methods developed by Etienne Fouvry--J\"urgen Kl\"uners, 
Alexander Smith and others. 

\smallskip
These results prove that the strong $\varepsilon$-conjecture for degree 
$p$ cyclic fields $K$ is true (in the meaning that this occurs with probability~$1$), 
except possibly for very sparse families of fields of density zero. One deduces
this from \cite[Theorem 1.1]{KP} by proving that for any map $h : \R \to \R$,
such that $h(X)\to\infty$ as $X\to\infty$, we have 
the following density property about the sets of
exceptional $p$-classes of the $p$-class groups:
$$\ds \lim_{X \to \infty} \bigg [ \frac{\order \big\{K: \  \sqrt{D_K} < X\ \ \&\ \ 
 \order \big(\Cl(K) \otimes \Z_p \big)^{1-\sigma}
> h (\sqrt{D_K}\,) \,\big\}} {\order \{K: \  \sqrt{D_K} < X \}} \bigg] = 0. $$

The possible infinite ``bad families'' of degree $p$ cyclic fields $K$ are such that
for some fixed $\varepsilon$ and for any $C>0$:
$$\order \big(\Cl(K) \otimes \Z_p \big)^{1-\sigma} > \ds
C\ \frac{\big(\sqrt{D_K}\big)^\varepsilon}{p^{N-1}} =:
C\,h_ {\varepsilon}(\sqrt{D_K}), $$ 

where $N$ is the number of ramified primes of $K$; 
we shall verify Section \ref{estimation} that these functions $h_{\varepsilon}$ 
fulfill the condition at infinity. 

\smallskip
But the Theorem \ref{thmf} shall prove (unconditionally) the $p$-rank 
$\varepsilon$-conjecture. 
\end{remark}

\subsubsection{The algorithm}
Recall briefly the computation of 
$\delta(N)$ and $\Delta(N)$ in the purpose of some probabilistic considerations.

\smallskip
Let $F$ be {\it any} degree $p$ cyclic  extension of $\Q$ with $N \geq 1$ ramified primes, 
put $M:= \Cl_F \otimes \Z_p$, and let $\sigma$ be a generator of 
$G := {\rm Gal}(F/\Q) \simeq \Z/p\Z$.

\smallskip
Let $I_F$ be the group of ideals of $F$, prime to $p$, and let 
$M_i =: \cl_F({\mathcal I}_i)$, $i \geq 0$, 
${\mathcal I}_i \subset I_F$, defined inductively by $M_0:= 1$, and:
$$M_{i+1}/M_i := (M/M_i)^{G}, \ \hbox{for $0\leq i \leq m-1$}, $$
where $m \geq 1 $ is the
least integer $i$ such that $M_i = M$ (i.e., such that $M_{i+1} = M_i$).
Then $M_i= \{c \in M, \, c^{(1-\sigma)^i} =1 \}$ for all $i \geq 0$ and 
$M_1 = M^G$.

\smallskip
The sequence $\order  (M_{i+1}/M_i)$, $0 \leq i \leq m$, decreases 
to $1$ and is bounded by $\order  M_1$ due to the injective maps
$M_{i+1}/M_i \hookrightarrow M_i/M_{i-1}
\hookrightarrow \cdots \hookrightarrow M_2/M_1 \hookrightarrow M_1$
defined by the operation of $1-\sigma$. 

\smallskip
We have, for the fields $F$, the formulas \cite[Corollary 3.7]{Gr3}, 
for all $i\geq 0$:
$$ \order  \big( M_{i+1} / M_i \big) = \frac{p^{N -1}}
{(\Lambda_i :  \Lambda_i \cap {\rm N}_{F/\Q}(F^\times))} \ \, \&
\ \, \Lambda_i := \{x \in \Q^\times,\  (x) \in 
{\rm N}_{F/\Q}({\mathcal I}_i) \}, $$

with ${\mathcal I}_0=1$ and $\Lambda_0=1$.

\smallskip
The progression of the algorithm depends on the $x \in \Lambda_i$,
$(x) =  {\rm N}_{F/\Q}({\mathfrak A})$, ${\mathfrak A} \in {\mathcal I}_i$,
such that $x= {\rm N}_{F/\Q}(y)$, $y \in F^\times$, giving the equation:
$$(y) = {\mathfrak A}\cdot {\mathfrak B}^{1-\sigma}, $$ 

in which the solutions ${\mathfrak B}$ (non-predictible) become new 
elements to be added to ${\mathcal I}_{i}$ to built ${\mathcal I}_{i+1}
\supseteq {\mathcal I}_{i}$, then $\Lambda_{i+1} \supseteq \Lambda_i$, 
and so on.

\smallskip
Put, for all $i\geq 0$, $\order \big(M_{i+1}/M_i \big) := p^{N-1-t_i^N}$, 
$t_i^N \geq 0$, where $p^{t_i^N} \mid p^{N -1}$ is the index 
$(\Lambda_i :  \Lambda_i \cap {\rm N}_{F/\Q}(F^\times))$ (since 
$x \in \Lambda_i$ is norm of an ideal, it is locally a norm everywhere, 
except perhaps at the $N$ ramified primes, but the product formula for Hasse's 
normic symbols gives the above divisibility).
We have:
$$t_0^N=0 \leq t_1^N \leq \cdots \leq t_m^N=N-1. $$

Then
$\order M= \order (\Cl_F \otimes \Z_p) = \prd_{i\geq 0} \order  \big( M_{i+1} / M_i \big)
= p_{}^{m\,(N-1) - \sum_{i?= 1}^m \,t_i^N}$.

We have the following result that we shall use to test the
$\varepsilon$-conjecture.

\begin{lemma}[{\cite[Lemma 4.2]{Gr3}}] Let $F/\Q$ be a degree $p$ cyclic  
extension with $N \geq 1$ ramified primes. Then ${\rm rk}_p(\Cl_F) =
(p-1) \cdot (N-1) - \sum_{i=1}^{p-2} t_i^N$.
Whence this yields $\delta(N) = (p-2)\cdot (N-1) - \sum_{i=1}^{p-2} t_i^N$.
\end{lemma}

We see that the $p$-rank may vary in the interval $[N-1, (p-1) \, (N-1)]$
and is always $O(N)$ as $N \to \infty$ which shall be
of a great importance; the $p$-rank is equal to $N-1$,
for all $F/\Q$ of degree $p$, if and only if $p=2$ as we have seen.
In the same manner, the $p^r$-ranks are given by the expressions:
$$\sm_{i=(r-1)\,(p-1)}^{r\,(p-1)-1} \big(N-1 - t_i^N \big) = 
(p-1) \cdot (N-1) - \sm_{i=(r-1)\,(p-1)}^{r\,(p-1)-1} t_i^N . $$

It is clear that the normic indices depend on the $\F_p$-ranks of 
$N \times N$-matrices of suitable Hilbert symbols \cite[Ch. VI, \S\,2]{Gr5} 
generalizing, for instance, R\'edei's matrices for the computation of the 
$4$-ranks of a quadratic field. 

\smallskip
But we emphasize the fact that if some heuristics on the $\F_p$-rank 
of these matrices are natural, {\it the number of steps of the algorithm} 
(i.e., $m$) only depends on distribution results (see Remark \ref{KP}).

\subsubsection{Examples of maximal rank}\label{rank6}
In fact, for $p>2$, the maximal rank,
equal to $(p-1) \cdot (N-1)$, is very rare. We remembered that in our thesis
\cite[p.\,39]{Gr5}, we gave the first example of cyclic cubic field, for which
${\rm rk}_3(\Cl) = 2\,(N-1)$; this example has conductor 
$f=9 \cdot 577 \cdot 757 \cdot 991$ (the program below 
shows that the property is fulfilled for the eight fields of conductor $f$):

\smallskip
\footnotesize
\begin{verbatim}
f=3895721091 N=4 P=x^3-1298573697*x+18010351461592  Cl=[9,3,3,3,3,3]
f=3895721091 N=4 P=x^3-1298573697*x-16034355152657  Cl=[3,3,3,3,3,3]
f=3895721091 N=4 P=x^3-1298573697*x+15707980296811  Cl=[3,3,3,3,3,3]
f=3895721091 N=4 P=x^3-1298573697*x-10132337699792  Cl=[3,3,3,3,3,3]
f=3895721091 N=4 P=x^3-1298573697*x-8835062576489   Cl=[3,3,3,3,3,3]
f=3895721091 N=4 P=x^3-1298573697*x-7829966535011   Cl=[9,3,3,3,3,3]
f=3895721091 N=4 P=x^3-1298573697*x+912031593193    Cl=[3,3,3,3,3,3]
f=3895721091 N=4 P=x^3-1298573697*x-221623244288    Cl=[9,3,3,3,3,3]
\end{verbatim}
\normalsize

\smallskip
We publish, for any further examples, a program computing
all the cyclic cubic fields such that there are exceptional $3$-classes; so in most cases the 
$3$-rank is $N$ (instead of $N-1$) and many examples have a non-trivial $9$-rank.

\smallskip
\footnotesize
\begin{verbatim}
{bf=7;Bf=10^6;for(f=bf,Bf,e=valuation(f,3);if(e!=0 & e!=2,next);F=f/3^e;
if(Mod(F,3)!=1||core(F)!=F,next);F=factor(F);Div=component(F,1);
d=component(matsize(F),1);for(j=1,d-1,D=component(Div,j);
if(Mod(D,3)!=1,break));for(b=1,sqrt(4*f/27),if(e==2 & Mod(b,3)==0,next);
A=4*f-27*b^2;if(issquare(A,&a)==1,if(e==0,if(Mod(a,3)==1,a=-a);
P=x^3+x^2+(1-f)/3*x+(f*(a-3)+1)/27);if(e==2,if(Mod(a,9)==3,a=-a);
P=x^3-f/3*x-f*a/27);N=omega(f);L=List;K=bnfinit(P,1);C8=component(K,8);
h=component(component(C8,1),1);Cl=component(component(C8,1),2);
dim=component(matsize(Cl),2);rho=0;for(i=1,dim,z=component(Cl,i);
v=valuation(z,3);if(v>=1,rho=rho+1;listput(L,3^v)));if(rho>=N,
print(" f=",f," N=",N," P=",P," Cl=",L)))))}

f=657 N=2 P=x^3-219*x-1241          Cl=[3,3]
f=657 N=2 P=x^3-219*x+730           Cl=[3,3]
(...)
f=3913 N=3 P=x^3+x^2-1304*x+17681   Cl=[3,3,3]
f=3913 N=3 P=x^3+x^2-1304*x-17536   Cl=[3,3,3]
(...)
f=4711 N=2 P=x^3+x^2-1570*x+19193   Cl=[9,3]
f=4711 N=2 P=x^3+x^2-1570*x-13784   Cl=[9,3]
f=4921 N=3 P=x^3+x^2-1640*x+23876   Cl=[3,3,3]
f=4921 N=3 P=x^3+x^2-1640*x-729     Cl=[3,3,3]
(...)
f=5383 N=2 P=x^3+x^2-1794*x+17744   Cl=[9,3]
f=5383 N=2 P=x^3+x^2-1794*x-9171    Cl=[9,3]
 (..)
f=15561 N=4 P=x^3-5187*x+141778     Cl=[9,3,3,3]
\end{verbatim}
\normalsize

\begin{remark}
The Galois action of $G = {\rm Gal}(F/\Q) \simeq \Z/p\Z$
on the non-$p$-part of the class group
gives rise, for all prime $q \mid \order \Cl_F$, to: 
$$\Cl_F \otimes \Z_q 
\simeq \Z[\mu_p]\big/ \prd_{{\mathfrak q} \mid q}
{\mathfrak q}^{n_{\mathfrak q}}, \ \hbox{$n_{\mathfrak q} \geq 0$}, $$

whose $\F_q$-dimension is a multiple of the
residue degree of $q$ in $\Q(\mu_p)/\Q$, which explains the rarity of such
divisibilities for large residue degrees.
\end{remark}

\subsection{Estimation of $C_{p,p, \varepsilon}$ 
for the fields $F_{N,p}$}\label{estimation}

As we have explained, we do not consider degree $p$ cyclic  extensions 
$F/\Q$ ramified at $p$. This shall modify the forthcoming computations 
by some $O(1)$ without any consequence on the statements since $p$
is fixed in all the sequel.

\smallskip
We need a lower bound of the $k$th prime number
$\ell_k \equiv 1 \pmod p$ to get an estimation of $f_{N,p}$. 

\smallskip
We thank G\'erald Tenenbaum for valuable indications for the good formula, from
\cite[Notes on Chapitre I, \S\,4.6]{T}, due to a result of Montgomery--Vaughan 
giving, for the $k$th prime number $\ell_k \equiv 1 \pmod p$:
\begin{equation}\label{rosser}
\ell_k >  \hbox{$\frac{p-1}{2}$} \cdot k\cdot {\rm log}\big(\hbox{$\frac{\ell_k}{p}$}\big) ,\ 
\hbox{for all $k\geq 1$}.
\end{equation}

Indeed, if $\pi(x;1,p) := \order \{\ell \leq x;\ \ell \equiv 1 \pmod p\}$ then:
$$\pi(x;1,p) \leq \ds \frac{2\,x}{(p-1)\, {\rm log} \big(\frac{x}{p} \big)}; $$ 
whence the result taking $x= \ell_k$ since $\pi(\ell_k;1,p)=k$.

\smallskip
We intend to test, for the family $(F_{N,p})_{N \geq 1}$
of discriminants $D_{N,p} := D_{F_{N,p}}$, the strong 
$\varepsilon$-conjecture, that is to say:
\begin{equation}\label{discriminant}
\order (\Cl_{F_{N,p}} \otimes \Z_p) =: p^{N-1+ \Delta(N)} \leq 
C_{p,p,\varepsilon} \cdot (\sqrt{D_{N,p}}\,)^\varepsilon,
\end{equation}

where $\Delta(N) = (m -1) \cdot (N-1) - \sum_{i = 1}^{m-1} t_i^N \geq 0$,
related to the set of exceptional $p$-classes, has only a probabilistic value 
depending on $m$ and the~$t_i^N$.

\smallskip
It will be easy, from the forthcoming calculations, to test the $p$-rank 
$\varepsilon$-conjecture, $\order (\Cl_{F_{N,p}} \otimes \F_p) \leq 
C_{p,p, \varepsilon} \cdot (\sqrt{D_{N,p}}\,)^\varepsilon$, but considering
instead the weaker inequality:
$$p^{N-1+ \delta(N)} = p^{(p-1) \,(N-1)- \sum_{i=1}^{p-2} t_i^N} \leq 
C_{p,p,\varepsilon} \cdot (\sqrt{D_{N,p}}\,)^\varepsilon, $$

even in the less favorable case $t_i^N = 0$, for $1 \leq i \leq p-2$
and replacing for all $N$, $D_{N,p}$ by a lower bound $D'_{N,p}$
(in other words the existence of an inequality
$p^{(p-1) \,(N-1)} \leq C'_{p,p,\varepsilon} \cdot (\sqrt{{D'}_{\!N,p}}\,)^\varepsilon$
proves the $p$-rank $\varepsilon$-conjecture for all degree $p$ cyclic fields).
We fix $p$ and put $D_{N,p}=:D_N =: f_N^{p-1}$. The strong form is equivalent to 
prove that $\ds \frac{\, p^{N-1 + \Delta(N)} \, }{(\sqrt{D_N}\,)^\varepsilon}$ 
is bounded as $N \to \infty$, whence
$(N-1 + \Delta(N)) \,{\rm log}(p) - \varepsilon \cdot 
\frac{p-1}{2} \sum_{k=1}^N {\rm log}(\ell_k) < \infty, \ \hbox{as $N \to \infty$}.$

\medskip
We then have, replacing $\ell_k$ by a lower bound $\ell'_k$, to compute,
using \eqref{chevalley}, \eqref{rosser}, \eqref{discriminant},
the following quantity:
\begin{equation*}
\begin{aligned}
 X(N) := &\ (N-1 + \Delta(N)) \,{\rm log}(p) -
 \varepsilon \cdot \hbox{$\frac{p-1}{2}$}
 \sm_{k=1}^N {\rm log}(\ell'_k) \\
= & \ (N-1 + \Delta(N)) \,{\rm log}(p) 
 - \varepsilon \cdot \hbox{$\frac{p-1}{2}$} 
\sm_{k=1}^N  {\rm log} \big(\hbox{$\frac{p-1}{2}$} \big) +  {\rm log} (k)
+\!  {\rm log}_2 \big (\hbox{$\frac{\ell_k}{p}$} \big)  .
\end{aligned}
\end{equation*}

We verify that $\sm_{k=1}^N {\rm log}_2 \big(\hbox{$\frac{\ell_k}{p}$} \big)$
can be neglected, subject to adding $-1$ to the sum, and consider, instead:
\begin{equation*}
\begin{aligned}
 X(N)= &\  (N-1 + \Delta(N)) \,{\rm log}(p)  
   - \varepsilon  \cdot \hbox{$\frac{p-1}{2}$} 
\Big[-1+\sm_{k=1}^N \big [ {\rm log} \big(\hbox{$\frac{p-1}{2}$} \big) 
+ {\rm log} (k) \big] \Big]   \\
 = &\  (N-1 + \Delta(N)) \,{\rm log}(p) -\varepsilon \cdot\hbox{$\frac{p-1}{2}$} 
\big[-1 + N \cdot {\rm log} \big(\hbox{$\frac{p-1}{2}$} \big) + {\rm log}(N !) \big].
\end{aligned}
\end{equation*}

\smallskip
 The expression of $N!$ leads to
${\rm log}(N!) = N {\rm log}(N)  - N + \hbox{$\frac{1}{2}$} {\rm log}(N) +O(1)$,
whence, with $\gamma_p := {\rm log} \big(\frac{p-1}{2} \big) -1$:
\begin{equation*}
\begin{aligned}
X(N) =&  (N-1 + \Delta(N))\, {\rm log}(p)   \\
 &\hspace{2cm}  - \varepsilon \cdot \hbox{$\frac{p-1}{2}$} \big[ N {\rm log}(N)
+N \cdot \gamma_p + \hbox{$\frac{1}{2}$} {\rm log}(N) + O(1) \big].
\end{aligned}
\end{equation*}

Now we write $X(N)$ under the form:
\begin{equation}\label{X}
\begin{aligned}
X(N) & = N \,{\rm log}(p) +\Delta(N) \,{\rm log}(p) 
- \varepsilon \cdot \hbox{$\frac{p-1}{2}$} N {\rm log}(N) \\
& \hspace{1.2cm} - {\rm log}(p) -\varepsilon \cdot \hbox{$\frac{p-1}{2}$} N \cdot \gamma_p
- \varepsilon \cdot \hbox{$\frac{p-1}{4}$}{\rm log}(N) 
- \varepsilon \cdot O(1) \\
& =  N \cdot 
\Big[ - \varepsilon \cdot \hbox{$\frac{p-1}{2}$}\, {\rm log}(N)
+ \hbox{$\frac{\Delta(N)}{N}$}\,{\rm log}(p)  \\
& \hspace{3.2cm} + {\rm log}(p) -  
\hbox{$\frac{{\rm log}(p)}{N}$} -\varepsilon \cdot  O(1) \Big] \\
& = N \cdot 
\Big[ - \varepsilon \cdot \hbox{$\frac{p-1}{2}$}\, {\rm log}(N) 
+ \hbox{$\frac{\Delta(N)}{N}$}\,{\rm log}(p) \\
  &\hspace{4.5cm} +  (1-N^{-1})\,{\rm log}(p) - \varepsilon \cdot  O(1)  \Big].
\end{aligned}
\end{equation}

Replacing $\Delta(N)$ by the maximal value $(p-2) \, (N-1)$ of $\delta(N)$,
the dominant term $- \varepsilon \cdot \frac{p-1}{2} {\rm log}(N)$ 
ensures the existence of a positive constant $C_\varepsilon$ 
since $\frac{(p-1) \, (N-1)}{N}\,{\rm log}(p)  = O(1)$ giving:
\begin{equation}\label{X0}
X_0(N)=- \varepsilon \, \hbox{$\frac{p-1}{2}$}\,N\, {\rm log}(N) +
N\,\big[ (p-1)\,{\rm log}(p) - \varepsilon \,O(1) - o(1)\big ].
\end{equation}

It is easy to verify that $X_0(N)$, as function of $N$, 
admits, for an $N_0 \gg 0$, a computable maximum, 
only depending on $p$ and $\varepsilon$ (e.g., for $p=7$,
$\varepsilon=0.1$, $N_0 \approx 2935394\cdot 10^{10}$,
$X_0(N_0) \approx 88 \cdot 10^{14}$).

\smallskip
This proves the $p$-rank $\varepsilon$-conjecture for the family 
$(F_{N,p})_{N \geq 1}$, even assuming always a maximal $p$-rank $(p-1)\,(N-1)$.
Whence we can state:

\begin{theorem}\label{thmf}
Let $p \geq 2$ be a given prime number. 

\smallskip
(i) The $p$-rank $\varepsilon$-conjecture on the existence,
for all $\varepsilon > 0$, of a constant $C_{d,p,\varepsilon}$ such that 
$\order (\Cl_K \otimes \F_p) \leq 
C_{d,p,\varepsilon} \cdot (\sqrt{D_K}\,)^\varepsilon$ for all $K$
of degree $d$, is true for the subfamily of all cyclic extension $F/\Q$ of degree $d=p$.

\smallskip
(ii) For any cyclic extension
$F/\Q$ of degree $p$, let $\Delta(N) \geq 0$ be defined by
$\order (\Cl_F \otimes \Z_p) = p^{N-1 + \Delta(N)}$, where $N$
is the number of ramified primes.
Then the strong $\varepsilon$-conjecture, $\order (\Cl_K \otimes \Z_p) 
\leq C_{d,p,\varepsilon} \cdot (\sqrt{D_K}\,)^\varepsilon$,
is true for the family of degree $p$ cyclic extensions under the condition 
$\Delta(N)\ll N\,{\rm log}(N)$.
\end{theorem}

\begin{proof}
Consider arbitrary prime numbers $\ell_{n_i} \equiv 1 \pmod p$, 
$1 \leq i \leq N$, and a field $K$, cyclic of degree $p$, of conductor $\prod_{i=1}^N \ell_{n_i}$. 
Thus $D_K \geq D_{F_N}$ and the maximal $p$-rank of $\Cl_K$ is still $(p-1)\,(N-1)$. 
We have seen that if we put $\delta(N)=(p-2)\,(N-1)$ in \eqref{X}, 
giving \eqref{X0}, the main term 
ensures the existence of the constant $C_{p,p,\varepsilon}$. 
We omit the details when $p^2$ (or $8$) divides the conductor.
\end{proof}

Very probably, considering the Remark \ref{KP},
$\Delta(N)$ is almost all the time of order much less than
$N\, {\rm log}(N)$ since 
$p^r$-ranks are in practice very rare for $r \geq 2$ as well as 
a maximal $p$-rank equal to $(p-1)\,(N-1)$ (see \S\,\ref{rank6});
but the dominant terms in \eqref{X} being $\Delta(N) \,{\rm log}(p) 
- \varepsilon \cdot \frac{p-1}{2}\,N\, {\rm log}(N)$, this
may give some trouble for the proof of an universal strong 
$\varepsilon$-conjecture assuming for instance 
$\Delta(N) = O(1)\, {\rm log}(N)$, or more, for infinite families
of degree $p$ cyclic fields, even of density zero.

\subsection{A lower bound for $C_{p,p,\varepsilon}$}
This part is not essential for our purpose, but it will show that
a lower bound of $C_{p,p,\varepsilon}$ is of the same
order of magnitude as for the upper bound deduced from \eqref{X0}. 
We shall use the following property which may be justified from
results given in \cite{May}. For all $k\geq 1$, the $k$th prime 
$\ell_k \equiv 1 \pmod p$ fulfills, for some constants $c_p$ 
only depending on $p$, the inequality
$\ell_k < c_p \, ((p-1)\,k) \cdot {\rm log}((p-1)\,k) <  c_p \, ((p-1)\,k)^2$.
Then: 
$$\begin{aligned}
& (N-1 + \delta(N)) \,{\rm log}(p)\ \leq \
 {\rm log}(C_\varepsilon) + \varepsilon \, \hbox{$\frac{p-1}{2}$}
\hbox{$\sum_{k=1}^N$} {\rm log}(\ell_k) \\
& \leq \  {\rm log}(C_\varepsilon) + \varepsilon \, \hbox{$\frac{p-1}{2}$}
\hbox{$\sum_{k=1}^N$} \big[{\rm log}(c_p) +2\, {\rm log}(p - 1) + 2\,{\rm log}(k) \big]  \\
& \leq \  {\rm log}(C_\varepsilon) + \varepsilon \, \hbox{$\frac{p-1}{2}$} \cdot 
\big[ N \, \gamma'_p  + 2 \,{\rm log}(N !) \big]
\end{aligned}, $$

where $\gamma'_p := {\rm log}(c_p) +  2\,{\rm log} (p-1) >0$. 
Whence, from the value of ${\rm log}(N !)$:
$$(N-1 + \delta(N)) {\rm log}(p)  \leq {\rm log}(C_\varepsilon) 
 + \varepsilon \hbox{$\frac{p-1}{2}$} \big[ 2 N {\rm log}(N) 
+N \, \gamma''_p  + {\rm log}(N)+O(1) \big],$$

with $\gamma''_p = \gamma'_p - 2>0$. Thus:
\begin{equation} \label{Y0}
\begin{aligned}
{\rm log}(C_\varepsilon) & \geq (N-1 + \delta(N)) \,{\rm log}(p)
- \varepsilon \, (p-1) \,N {\rm log}(N) \\
& \hspace{2cm} -\varepsilon \, \hbox{$\frac{p-1}{2}$} N \, \gamma''_p
- \varepsilon \, \hbox{$\frac{p-1}{2}$} {\rm log}(N) - \varepsilon\, O(1) \\
& \geq N \cdot \Big[ - \varepsilon \, (p-1)\, {\rm log}(N)
+ \hbox{$\frac{\delta(N)}{N}$}\,{\rm log}(p) + O(1)\Big] .
\end{aligned}
\end{equation}

\begin{remarks} (i) As soon as we replace
$\varepsilon$ by $c+\varepsilon$, $0 < c <1$, as it is done in much papers 
giving general proofs for $\order (\Cl_K \otimes \F_p) \leq C_{d,p,c,\varepsilon} 
\cdot (\sqrt{D_K}\,)^{c+\varepsilon}$, the above computations 
becomes, replacing $\Delta(N)$ by $\delta(N)$:
\begin{equation*}
\begin{aligned}
{\rm log}(C_{p,p,c,\varepsilon})  \leq & \ N \cdot 
\Big[ - (c+\varepsilon) \cdot\hbox{$\frac{p-1}{2}$}\, {\rm log}(N) 
+ \hbox{$\frac{\delta(N)}{N}$}\,{\rm log}(p) + O(1) \Big] \\
{\rm log}(C_{p,p,c,\varepsilon})  \geq 
& \ N \cdot \Big[- (c+ \varepsilon) \cdot (p-1)\, {\rm log}(N) 
+ \hbox{$\frac{\delta(N)}{N}$}\,{\rm log}(p) + O(1) \Big]. 
\end{aligned}
\end{equation*}

For $\varepsilon \to 0$, the formulas \eqref{X0} and 
\eqref{Y0} give a limit value $C_{p,c}$ of $C_{p,p,c,\varepsilon}$ such that
${\rm log}(C_{p,c}) \approx (N-1 + \delta(N))\,{\rm log}(p) 
- c\, \frac{p-1}{2}\,N \,{\rm log}(N) - N\,O(1)$, rapidely negative as 
$N$ increases, giving for degree $p$ cyclic fields, an obvious proof
of the ``$p$-rank $(c+\varepsilon)$-property''.

\smallskip
(ii) If we replace $\Q$ by a number field $k$ and the 
fields $F_{N,p}$ by the degree $p$ cyclic  extensions $F_{k,N,p}$ of $k$ 
with $N$ ramified prime ideals of $k$, the details of computations 
are more complicate, but the results and comments are similar
since the $p$-rank of $\Cl_{F_{k,N,p}}$ 
is still equivalent to $O(N)$ because of the exact 
sequence of genus theory and the inequalities \cite[Theorem IV.4.5.1]{Gr1}:
$$N - c_{k,p} \leq {\rm rk}_p \big(\Cl_{F_{k,N,p}}^G \big) \leq N+ c'_{k,p}, $$

where the constants $c_{k,p}$, $c'_{k,p}$ depend on the $p$-classes 
and units of $k$.
Then, the general algorithm computing the $p$-rank via the filtration 
$(M_i)_{i \geq 0}$ is identical up to similar modifications (see \cite[\S\,4.4]{Gr3}) 
and shall give the $p$-rank $\varepsilon$-conjecture for the relative degree $p$ cyclic  
case without too much difficulties. 
The case of abelian $p$-extensions may be 
accessible from the $p^n$-cyclic cases, with some effort...

\smallskip
The proof of the strong $\varepsilon$-conjecture, for the degree $p$ cyclic 
extensions of $k$, remains, theoretically, open, but is clearly not a folk conjecture for 
such very particular real fields because of the possible generalization of the density 
results of \cite{KP} and the conclusion, about the possible existence of
pathological families of density zero, is still relevant.
Meanwhile a particular study of the algorithm giving $\Delta(N)$, independently
of any density results, should be a crucial step for many questions in number theory.
\end{remarks}

\section{PARI/GP programs computing $\Cl_{F_{N,p}}$}
\subsection{Structure of some $\Cl_{F_{N,p}}$}
The following numerical results show that exceptional $p$-classes (indicated by ${}^*$)
are not excessively frequent for these fields. We examine the cases $p=3, 5, 7$, 
then $p=2$ (one must precise $p$ and the conductor $f=\ell_1\cdots \ell_N$,
$\ell_i \equiv 1 \pmod p$, in 
the following program giving all the $(p-1)^{N-1}$ fields of conductor $f$):

\footnotesize
\begin{verbatim}
{p=3;f=7*13*19*31*37;V=polsubcyclo(f,p);d=matsize(V);d=component(d,2);
for(k=1,d,P=component(V,k);if(nfdisc(P)!=f^(p-1),next);K=bnfinit(P,1);
C8=component(K,8);C81=component(C8,1);h=component(C81,1);
Cl=component(C81,2);print("p=",p," f=",f," P=",P," Cl=",Cl))}

p=3  f=1983163=7*13*19*31*37
P=x^3+x^2-661054*x+ 49725976    Cl=[6,6,3,3]
P=x^3+x^2-661054*x+198463201    Cl=[3,3,3,3]
P=x^3+x^2-661054*x+97321888     Cl=[39,3,3,3,3]*
P=x^3+x^2-661054*x-186270421    Cl=[3,3,3,3]
P=x^3+x^2-661054*x-79179619     Cl=[6,6,3,3]
P=x^3+x^2-661054*x-188253584    Cl=[3,3,3,3,3]*
P=x^3+x^2-661054*x+138968311    Cl=[3,3,3,3]
P=x^3+x^2-661054*x-146607161    Cl=[3,3,3,3,3]*
P=x^3+x^2-661054*x-158506139    Cl=[3,3,3,3]
P=x^3+x^2-661054*x-140657672    Cl=[6,6,3,3]
P=x^3+x^2-661054*x+186564223    Cl=[3,3,3,3]
P=x^3+x^2-661054*x+81456584     Cl=[3,3,3,3]
P=x^3+x^2-661054*x+206395853    Cl=[3,3,3,3]
P=x^3+x^2-661054*x-206102051    Cl=[12,12,3,3]
P=x^3+x^2-661054*x+2130064      Cl=[3,3,3,3,3]*
P=x^3+x^2-661054*x+27911183     Cl=[3,3,3,3,3]*

p=3  f=85276009=7*13*19*31*37*43
P=x^3-x^2-28425336*x+58104545836    Cl=[3,3,3,3,3]
P=x^3-x^2-28425336*x+56995957719    Cl=[3,3,3,3,3]
P=x^3-x^2-28425336*x-7472705085     Cl=[3,3,3,3,3]
P=x^3-x^2-28425336*x+10264704787    Cl=[3,3,3,3,3]
P=x^3-x^2-28425336*x+51623569152    Cl=[3,3,3,3,3]
P=x^3-x^2-28425336*x+30901498965    Cl=[3,3,3,3,3]
P=x^3-x^2-28425336*x-50622365639    Cl=[6,6,3,3,3]
P=x^3-x^2-28425336*x-29132811371    Cl=[3,3,3,3,3]
P=x^3-x^2-28425336*x-46614393216    Cl=[3,3,3,3,3]
P=x^3-x^2-28425336*x-33226059803    Cl=[3,3,3,3,3]
P=x^3-x^2-28425336*x-56506410260    Cl=[3,3,3,3,3]
P=x^3-x^2-28425336*x-12248161589    Cl=[3,3,3,3,3]
P=x^3-x^2-28425336*x+23994142236    Cl=[6,6,3,3,3]
P=x^3-x^2-28425336*x+52220501215    Cl=[3,3,3,3,3,3]*
P=x^3-x^2-28425336*x+11799672949    Cl=[6,6,3,3,3]
P=x^3-x^2-28425336*x-57529722368    Cl=[3,3,3,3,3]
P=x^3-x^2-28425336*x-21287418543    Cl=[3,3,3,3,3]
P=x^3-x^2-28425336*x-47040773261    Cl=[3,3,3,3,3]
P=x^3-x^2-28425336*x+36273887532    Cl=[3,3,3,3,3]
P=x^3-x^2-28425336*x-53436473936    Cl=[6,6,3,3,3,3]*
P=x^3-x^2-28425336*x+40964068027    Cl=[21,3,3,3,3]
P=x^3-x^2-28425336*x+45824800540    Cl=[3,3,3,3,3]
P=x^3-x^2-28425336*x+56313749647    Cl=[3,3,3,3,3]
P=x^3-x^2-28425336*x+13078813084    Cl=[3,3,3,3,3]
P=x^3-x^2-28425336*x+39940755919    Cl=[3,3,3,3,3]
P=x^3-x^2-28425336*x-37830964289    Cl=[3,3,3,3,3]
P=x^3-x^2-28425336*x+57848717809    Cl=[6,6,3,3,3]
P=x^3-x^2-28425336*x-9775157328     Cl=[3,3,3,3,3]
P=x^3-x^2-28425336*x-2526696563     Cl=[3,3,3,3,3]
P=x^3-x^2-28425336*x-58126654431    Cl=[3,3,3,3,3]
P=x^3-x^2-28425336*x+28428494704    Cl=[3,3,3,3,3]
P=x^3-x^2-28425336*x-51986781783    Cl=[3,3,3,3,3]

p=5  f=13981=11*31*41
P=x^5+x^4-5592*x^3+46417*x^2+4301003*x-26664769     Cl=[5,5]
P=x^5+x^4-5592*x^3+46417*x^2+2623283*x-26664769     Cl=[5,5]
P=x^5+x^4-5592*x^3-261165*x^2-4479065*x-26832541    Cl=[55,5,5]*
P=x^5+x^4-5592*x^3+32436*x^2+1518784*x+1814528      Cl=[5,5]
P=x^5+x^4-5592*x^3-205241*x^2-2074333*x-6028813     Cl=[5,5,5]*
P=x^5+x^4-5592*x^3-135336*x^2+847696*x+1143440      Cl=[5,5]
P=x^5+x^4-5592*x^3-205241*x^2-1878599*x+234675      Cl=[5,5]
P=x^5+x^4-5592*x^3+74379*x^2+3993421*x-7035445      Cl=[5,5]
P=x^5+x^4-5592*x^3+158265*x^2-424575*x-12851541     Cl=[5,5]
P=x^5+x^4-5592*x^3+130303*x^2+442247*x+346523       Cl=[5,5]
P=x^5+x^4-5592*x^3+46417*x^2+1644613*x-16878069     Cl=[5,5]
P=x^5+x^4-5592*x^3+214189*x^2-2493763*x+5715227     Cl=[5,5]
P=x^5+x^4-5592*x^3+74379*x^2+3238447*x-28174717     Cl=[5,5]
P=x^5+x^4-5592*x^3+32436*x^2+5992704*x-2659392      Cl=[5,5,5]*
P=x^5+x^4-5592*x^3-37469*x^2+4384889*x-11397517     Cl=[5,5]
P=x^5+x^4-5592*x^3-9507*x^2+7432747*x+1185383       Cl=[5,5]
\end{verbatim}
\normalsize

\subsection{Examples of exeptional $p$-classes}
If one drops the invariant classes from the class group, 
this gives a very small component (often equal to $1$); for instance we list
some fields of the above tables with exceptional $p$-classes and/or nontrivial
non-$p$-parts of the class group (for $p=3, 5, 7, 11$):

\smallskip
\footnotesize
\begin{verbatim}
p=3
f=1983163  P=x^3+x^2-661054*x+97321888 Cl=[39,3,3,3,3]=[13]x[3,3,3,3,3]*
f=1983163  P=x^3+x^2-661054*x-206102051Cl=[12,12,3,3]=[4,4]x[3,3,3,3]
f=85276009 P=x^3-x^2-28425336*x-53436473936 
                                       Cl=[6,6,3,3,3,3]=[2,2]x[3,3,3,3,3,3]*                        
p=5
f=13981    P=x^5+x^4-5592*x^3-261165*x^2-4479065*x-26832541
                                       Cl=[55,5,5]=[11]x[5,5,5]*
f=13981    P=x^5+x^4-5592*x^3-205241*x^2-2074333*x-6028813 
                                       Cl=[5, 5, 5]*
f=13981    P=x^5+x^4-5592*x^3+32436*x^2+5992704*x-2659392 
                                       Cl=[5, 5, 5]*
p=7
f=88537   P=x^7+x^6-37944*x^5-1134719*x^4+324123632*x^3
+13095064100*x^2-393352790753*x-8536744545007
                                       Cl=[301,7]=[43]x[7,7]
f=10004681   P=x^7-x^6-4287720*x^5-1266715121*x^4+5127549957760*x^3
+2650344024068794*x^2-951078919604894529*x-488606218944681147667
                                       Cl=[791,7,7]=[113]x[7,7,7]
f=10004681   P=x^7-x^6-4287720*x^5+3455494311*x^4+176273349584*x^3
-471685834336278*x^2-36087798097778993*x+6487901368894795147
                                       Cl=[7,7,7,7]*
f=10004681   P=x^7-x^6-4287720*x^5+2985274304*x^4+1830007100160*x^3
-1788009977372784*x^2+233770322355404864*x-5599630780142239232
                                       Cl=[14,14,14]=[2,2,2]x[7,7,7]
f=10004681   P=x^7-x^6-4287720*x^5+6276814353*x^4-4047542893720*x^3
+1360785664233294*x^2-232635292693132049*x+15953699891990750023
                                       Cl=[7,7,7,7]*
f=10004681   P=x^7-x^6-4287720*x^5-6169008811*x^4-3865457699520*x^3
-1247559831026016*x^2-202130636944756129*x-12969698603184144677 
                                       Cl=[7,7,7,7]*
f=10004681   P=x^7-x^6-4287720*x^5+1224450448*x^4+5079527488960*x^3
-2161274540764336*x^2-1677161713287529664*x+812143879436422435328 
                                       Cl=[7,7,7,7]*
f=10004681   P=x^7-x^6-4287720*x^5-1306733845*x^4+4849019638720*x^3
+2420349994235592*x^2-659732951886568641*x-207964718993797238079 
                                       Cl=[7,7,7,7]*
f=10004681   P=x^7-x^6-4287720*x^5+5756570941*x^4-3107102879720*x^3
+760318637129246*x^2-74394888056758073*x+1594979915105904419 
                                       Cl=[7,7,7,7]
f=10004681   P=x^7-x^6-4287720*x^5-1746939809*x^4+2392910471944*x^3
+1648023037138232*x^2+241022177190387487*x-9669121620934915453 
                                       Cl=[7,7,7,7]*
f=10004681   P=x^7-x^6-4287720*x^5-4188081973*x^4-284662313448*x^3
+706702291060440*x^2+90226184532822239*x-3998693323498243787 
                                       Cl=[7,7,7,7]*
f=10004681   P=x^7-x^6-4287720*x^5-2227164497*x^4+1569805356712*x^3
+369664737597974*x^2-166416491455189217*x+3813962316737479895 
                                       Cl=[7,7,7,7]*
f=10004681   P=x^7-x^6-4287720*x^5-3307670045*x^4+1296557509240*x^3
+1587033521303494*x^2+232482552302284071*x-10617355312468435915 
                                       Cl=[203,7,7]=[203]x[7,7,7]
f=10004681   P=x^7-x^6-4287720*x^5+1654651731*x^4+4187910318240*x^3
-1965195178959414*x^2-1012216801097102473*x+451605062388713519719 
                                       Cl=[14,14,14]=[2,2,2]x[7,7,7]
f=10004681   P=x^7-x^6-4287720*x^5+1544600240*x^4+4818925558272*x^3
-2343209264562096*x^2-1609606088655340480*x+886992887186768275456 
                                       Cl=[203,7,7]=[203]x[7,7,7]
f=10004681   P=x^7-x^6-4287720*x^5-1306733845*x^4+4427622475000*x^3
+1798213808544282*x^2-1313378138040048441*x-613859706212867719587 
                                       Cl=[301,7,7]=[301]x[7,7,7]
p=11
f=137149   P=x^11+x^10-62340*x^9-2099173*x^8+998038116*x^7+30321726924*x^6
-5078707527329*x^5+3275334221180*x^4+8066546096404148*x^3
-214723422858644515*x^2+1902599837479513519*x-4121588229203611219
                                       Cl=[253,11]=[23]x[11,11]
f=137149   P=x^11+x^10-62340*x^9-727683*x^8+1217887963*x^7+33409088063*x^6
-7760886906947*x^5-350751766601032*x^4+3398347545513222*x^3
+236507399684756272*x^2+2593585988882665302*x+8529384350363670191
                                       Cl=[979,11]=[89]x[11,11]
f=137149   P=x^11+x^10-62340*x^9-1550577*x^8+1265615815*x^7+66889353347*x^6
-7866931610939*x^5-712972865698216*x^4-12900860936489076*x^3
+143276981594922336*x^2+1596818122984871186*x+3178173588229813309
                                       Cl=[737,11]=[67]x[11,11]
f=137149   P=x^11+x^10-62340*x^9-7310835*x^8+70636578*x^7+43296296622*x^6
+1378934348258*x^5-28471672310749*x^4-944763467217249*x^3
+12433374265353417*x^2+31266667418235948*x-324164722946199831
                                       Cl=[253,11]=[23]x[11,11]
f=137149   P=x^11+x^10-62340*x^9-727683*x^8+1153290784*x^7+18055120364*x^6
-7130815486262*x^5-134781314432095*x^4+13777568483843493*x^3
+310799275320778321*x^2+883774494827373474*x-5728549445587601897
                                       Cl=[253,11]=[23]x[11,11]
\end{verbatim}
\normalsize

For $p=2$, $F_{N,2}$ may be real or complex; as we know, the $2$-rank is always $N-1$ and 
any exceptional classes give non-trivial $4$-ranks. The PARI/GP instruction 
${\sf bnfnarrow}$ allows the $2$-structure in the restricted sense:

\footnotesize
\begin{verbatim}
{m=1;for(N=2,100,el=prime(N);m=(-1)^((el-1)/2)*el*m;P=x^2-m;
K=bnfinit(P,1);L=bnfnarrow(K);Cl=component(L,2);print("m=",m," Cl=",Cl))}
m=-15  Cl=[2]
m=+105  Cl=[2,2]
m=-1155  Cl=[2,2,2]
m=-15015  Cl=[12,2,2,2]
m=-255255  Cl=[16,2,2,2,2]*
m=+4849845  Cl=[4,2,2,2,2,2]*
m=-111546435  Cl=[42,2,2,2,2,2,2]
m=-3234846615  Cl=[308,2,2,2,2,2,2,2]*
m=+100280245065  Cl=[2,2,2,2,2,2,2,2,2]
m=+3710369067405  Cl=[34,2,2,2,2,2,2,2,2,2]
m=+152125131763605  Cl=[2,2,2,2,2,2,2,2,2,2,2]
m=-6541380665835015  Cl=[28284,2,2,2,2,2,2,2,2,2,2,2]*
m=+307444891294245705  Cl=[14,2,2,2,2,2,2,2,2,2,2,2,2]
m=+16294579238595022365  Cl=[2,2,2,2,2,2,2,2,2,2,2,2,2,2]
m=-961380175077106319535  Cl=[1210734,2,2,2,2,2,2,2,2,2,2,2,2,2,2]
m=-58644190679703485491635  Cl=[1622526,2,2,2,2,2,2,2,2,2,2,2,2,2,2,2]
m=+3929160775540133527939545 Cl=[2,2,2,2,2,2,2,2,2,2,2,2,2,2,2,2,2]
m=-278970415063349480483707695 
                   Cl=[83911452,2,2,2,2,2,2,2,2,2,2,2,2,2,2,2,2,2]*
m=-20364840299624512075310661735 
                   Cl=[362626834,2,2,2,2,2,2,2,2,2,2,2,2,2,2,2,2,2,2]
m=+1608822383670336453949542277065 
                   Cl=[4,2,2,2,2,2,2,2,2,2,2,2,2,2,2,2,2,2,2,2]*
m=-133532257844637925677812008996395 
                   Cl=[2322692420,2,2,2,2,2,2,2,2,2,2,2,2,2,2,2,2,2,2,2,2]*
\end{verbatim}
\normalsize

\section{Genus theory -- Abelian $p$-ramification theory}

\subsection{Reminders on Genus theory}

For wide information on genus theory, see for example 
\cite{Fr2,Fu,Gr1,J2,Mai,Ra}. 

\subsubsection{Genus field and genus number $g_{K/k}$}
Introduce the genus field according to an extension $K/k$
(we shall take $k = \Q$ in the sequel) which is the maximal subextension 
$H_{K/k}$ of $H_K$ (in the restricted sense) equal 
to the compositum of $K$ with an abelian extension of~$k$:
\vspace{0.3cm}\par
\unitlength=0.75cm
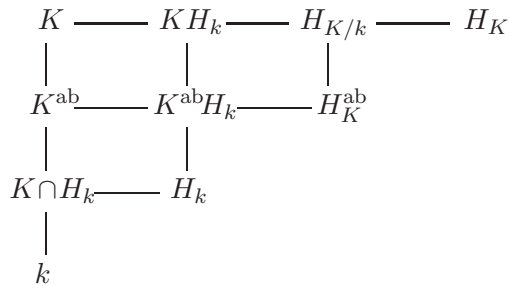
$$\vbox{\hbox{\hspace{0.5cm} \begin{picture}(11.5,4.6)
\put(7.35,5.0){\line(1,0){1.3}}
\put(4.7,5.0){\line(1,0){1.1}}
\put(2.0,5.0){\line(1,0){1.3}}
\put(4.9,3.5){\line(1,0){1.3}}
\put(2.0,3.5){\line(1,0){1.3}}
\put(2.35,2.0){\line(1,0){1.15}}
\put(1.5,3.9){\line(0,1){0.75}}
\put(1.5,2.4){\line(0,1){0.75}}
\put(1.5,0.9){\line(0,1){0.75}}
\put(4.00,3.9){\line(0,1){0.75}}
\put(4.00,2.4){\line(0,1){0.75}}
\put(6.5,3.9){\line(0,1){0.75}}
\put(8.9,4.9){$H_K$}
\put(6.0,4.9){$H_{K/k}$}
\put(3.5,4.9){$KH_k$}
\put(1.35,4.9){$K$}
\put(6.3,3.4){$H_K^{\rm ab}$}
\put(1.2,3.4){$K^{{\rm ab}}$}
\put(3.7,1.9){$H_k$}
\put(0.85,1.9){$K\!\cap\!H_k$}
\put(3.4,3.4){$K^{{\rm ab}}\! H_k$}
\put(1.3,0.40){$k$}
\end{picture} }} $$
\unitlength=1.0cm
Thus $H_{K/k}$ is for instance equal to the
compositum of $K$ with $H_K^{\rm ab}$ (the maximal abelian 
subextension of $H_K/k$), according to the diagram above,
where $K^{\rm ab}$ is the maximal abelian subextension 
of $K/k$. The genus number is $g_{K/k} := [H_{K/k} :K H_k]$.
When $K/k$ is cyclic, the genus number $g_{K/k}$
is equal to the number of invariant classes by ${\rm Gal}(K/k)$
given by Chevalley's formula \eqref{chevalley}.
In the general Galois case, we have the similar expression
$g_{K/k} = \ds \frac{\order \Cl_k \cdot
\prod_{\mathfrak l}\, e^{\rm ab}_{\mathfrak l}}{ [K^{\rm ab} : k] \cdot
(E_k^{\rm pos} : E_k^{\rm pos} \cap {\mathcal N}_{K/k})},$
where $E_k^{\rm pos}$ is the group of {\it totally positive} units of $k$, 
$e^{\rm ab}_{\mathfrak l}$ the index of
ramification of ${\mathfrak l}$ in $K^{\rm ab}/k$
and ${\mathcal N}_{K/k}$ the group of local norms in $K/k$.
A general formula does exist for non-Galois fields
(e.g.,\cite[Theorem IV.4.2 \& Corollaries]{Gr1}).

\subsubsection{Variants of the strong $\varepsilon$-conjecture}\label{bigh}
Since the $p$-genus group of $K$ may be an obstruction to the 
strong $\varepsilon$-conjecture, we may consider, in the exact sequence 
$1 \to \Cl'_K \otimes \Z_p \to \Cl_K \otimes \Z_p 
\to {\rm Gal}(H_{K/\Q}/ K)\otimes \Z_p \to 1$, the 
number $\order (\Cl'_K \otimes \Z_p)$ (giving the number of exceptional 
$p$-classes instead of the whole $p$-class group) and propose the following
form of the $\varepsilon$-conjecture:

\begin{conjecture} \label{epsconj}
For a number field $K$, let $H_K$ be its Hilbert's class field, $H_{K/\Q}$ 
its genus field and $\Cl'_K := {\rm Gal}(H_K/H_{K/\Q})$. Let $p$ be a prime 
number. For all $\varepsilon > 0$ there exists  $C'_{d,p,\varepsilon}$ such 
that $\order( \Cl'_K  \otimes \Z_p) \leq C'_{d,p,\varepsilon} \cdot 
(\sqrt{D_K}\,)^{\varepsilon}$ holds for all $K$ 
of degree $d$, except possibly for sparse families of density zero.
\end{conjecture}

One may ask what happens for a ``global'' strong $\varepsilon$-conjecture
on the form:
$$\order \Cl_K \leq \wh C_{d,\varepsilon} \cdot (\sqrt{D_K}\,)^{\varepsilon}.$$

The paper \cite{Da} of Ryan Daileda recalls some results by Littlewood 
showing that (under GRH) there
exist imaginary quadratic fields with arbitrary large discriminant 
for which $\order\Cl_K \geq c \cdot \sqrt{D_K} \,{\rm log}_2(D_K)$,
where $c$ is an absolute constant. For real quadratic fields a result
of Montgomery and Weinberger is that there exist real quadratic fields 
with arbitrary large discriminant whose class numbers satisfy
$\order\Cl_K \geq c \cdot \sqrt{D_K} \,\frac{{\rm log}_2(D_K)}{{\rm log}(D_K)}$.
Analogous results are known for cyclic cubic fields and Daileda proves
that there exists an absolute constant $c>0$ so that there are
totally real non-abelian cubic fields, with arbitrary large discriminant
satisfying $\order\Cl_K \geq c \cdot \sqrt{D_K} \,
\Big( \frac{{\rm log}_2(D_K)}{{\rm log}(D_K)} \Big)^2$.
All this has been generalized to CM number fields in \cite{DaKM}.

\smallskip
So, if we consider an inequality of ``strong global $\varepsilon$-conjecture'' type:
$$\order \Cl_K \leq \wh C_{d,\varepsilon} \cdot (\sqrt{D_K})^\varepsilon, 
\ \ \hbox{for all $K$ of degree $d$, }$$
any {\it infinite family} ${\mathcal K}$ of
fiields $K$ such that $\order \Cl_K \geq c \cdot\sqrt{D_K}$, 
for a constant $c > 0$, independent of $K \in {\mathcal K}$, yields:
$${\rm log}(\wh C_{d,\varepsilon}) \geq {\rm log}(c)+
(1- \varepsilon) \cdot {\rm log}(\sqrt{D_K}\,), $$ 

which is absurd. In other words, there is in general
no strong global $\varepsilon$-conjecture; nevertheless the question of the 
strong form we have considered (for $p$ fixed and for $\Cl_K$ or $\Cl'_K$):
$$\order(\Cl_K \otimes \Z_p) \leq 
C_{d,p, \varepsilon} \cdot (\sqrt{D_K}\,)^ \varepsilon, $$ 

depends on the 
finiteness of fields $K$ such that $\order(\Cl_K \otimes \Z_p) \geq c_p \cdot \sqrt{D_K}$.
For instance, in the quadratic case and $p>2$, this should give some
$p$-class groups for which the $p$-rank and/or the exponent tend to infinity 
with $D_K$; even in the imaginary case, this may occur only for very sparse families.

\subsubsection{Computation of some successive local maxima}
The following program, for imaginary quadratic fields,
allows a study of this question by computing the 
successive local maxima of $C'_{2,\varepsilon}$, in $\sf{C}$,
the discriminant and the class number obtained for each local 
maximum, in $\sf{D}$ and $\sf{h}$, respectively:

\smallskip
\footnotesize
\begin{verbatim}
{eps=0.05;Cm=0;bD=2;BD=10^9;for(D=bD,BD,e=valuation(D,2);M=D/2^e;
if(core(M)!=M,next);if((e==1||e>3)||(e==0&Mod(M,4)!=-1)||(e==2&Mod(M,4)!=1),
next);h=qfbclassno(-D);N=omega(D);C=h/(2^(N-1)*(sqrt(D)^eps));
if(C>Cm,Cm=C;print("D=",-D," h=",h," C=",C)))}
\end{verbatim}
\normalsize

\smallskip
As far as the program has run, it slows down between
$\sf{C = 6\cdot 10^4}$ and $\sf{C = 7 \cdot 10^4}$, but no conclusion
is possible as expected from the above.
But the most spectacular fact, that we have discover, is that, for each local 
maximum $\sf{C}$, the corresponding discriminant is prime (whence $\sf{h}$ odd)
whatever $\varepsilon$, as shown by this short excerpt:

\footnotesize
\begin{verbatim}
 D           h       C
-3           1       0.972908434869468710702241668941166407
-23          3       2.773818617890694606606085132125197163
-47          5       4.541167885124564220325740509229014479
-71          7       6.292403751297605635733619062872115785
-167         11      9.678872599268429560299054329160821597
-191         13     11.400332501352005304200415816510168367
-239         15     13.080709822134822456679612679136456819
-311         19     16.460180420909375330798097085967676763
-431         21     18.045019802162182082161592477498679286
-479         25     21.425532320359474690178248184779886979
(...)
-1118314391  77395  45972.572539103313552220923397893022055
-1130984399  77697  46138.963607764612827282941613333974355
-1139075159  78141  46394.356488699687700451681236043961611
-1184068679  81267  48203.636807716833174323833829210758624
-1229647319  81419  48248.214898940545998582986401841694028
-1237871879  82171  48685.729279564354497202266261569081321
-1250370239  83503  49462.505651262222727040029721940389296
\end{verbatim}
\normalsize

We have no counterexample in the selected interval $D \leq 2\cdot 10^9$ and no 
serious explanation, but if we test the local maxima of  $\frac{\order \Cl_K}
{(\sqrt{D_K}\,)^{\varepsilon}}$ instead of $\frac{2^{-(N-1)}\,\order \Cl_K}
{(\sqrt{D_K}\,)^{\varepsilon}}$, we have for instance the following 
normal behaviour:

\footnotesize
\begin{verbatim}
D=[3, 1; 5, 1]           h=2     C=1.8690792417830016333288729232969402612
D=Mat([23, 1])           h=3     C=2.7738186178906946066060851321251971632
D=[3, 1; 13, 1]          h=4     C=3.6499202570298105632113852382187054319
D=Mat([47, 1])           h=5     C=4.5411678851245642203257405092290144795
D=Mat([71, 1])           h=7     C=6.2924037512976056357336190628721157852
D=[5, 1; 19, 1]          h=8     C=7.1391564091327201771370767959908455517
D=[7, 1; 17, 1]          h=10    C=8.8738345254872857598506328581670649858
(...)
D=Mat([63839, 1])        h=423   C=320.78438186581184951419633278279394585
D=[23, 1; 2833, 1]       h=424   C=321.37826092893978747476207092089453868
D=[113, 1; 607, 1]       h=434   C=328.53606501556853121813362584226004750
D=[19, 1; 23, 1; 163, 1] h=440   C=332.76370450428389891215784283045533831
D=[41, 1; 1831, 1]       h=454   C=342.90123385517606193005461615124054801
D=[19, 1; 37, 1; 113, 1] h=468   C=352.97586311634473415004028909597003881
D=[43, 1; 1973, 1]       h=480   C=361.43179021472306392157190052394995551
D=Mat([88079, 1])        h=487   C=366.35924203761999218970143789845358049
\end{verbatim}
\normalsize

In the same way, if we compute the successive maxima of the $3$-class
groups, we obtain a similar result:

\footnotesize
\begin{verbatim}
{p=3;bD=1;BD=10^9;Cm=0;for(D=bD,BD,e=valuation(D,2);M=D/2^e;
if(core(M)!=M,next);if((e==1||e>3)||(e==0&Mod(M,4)!=-1)||(e==2&Mod(M,4)!=1),
next);h=qfbclassno(-D);hp=p^valuation(h,p);Cp=hp;if(Cp>Cm,Cm=Cp;
C=log(Cp)/log(sqrt(D));print("D=",-D," h=",h," hp=",hp," C=",C)))}
D=-23         h=3     hp=3      C=0.70075861284442195481324
D=-199        h=9     hp=9      C=0.83019007976763598642971
D=-983        h=27    hp=27     C=0.95661698654993161545339
D=-3671       h=81    hp=81     C=1.07074359233325762042197
D=-29399      h=243   hp=243    C=1.06778367209896382287404
D=-178559     h=729   hp=729    C=1.09019287826209803280171
D=-2102999    h=2187  hp=2187   C=1.05643959875714455523718
D=-14868719   h=6561  hp=6561   C=1.06436822551851827813563
D=-98311919   h=19683 hh=19683  C=1.07451592116950263349372
\end{verbatim}
\normalsize

Then, for $p=2$:

\footnotesize
\begin{verbatim}
{p=2;bD=1;BD=10^9;Cm=0;for(D=bD,BD,e=valuation(D,2);M=D/2^e;
if(core(M)!=M,next);if((e==1||e>3)||(e==0&Mod(M,4)!=-1)||(e==2&Mod(M,4)!=1),
next);h=qfbclassno(-D);hp=p^valuation(h,p);Cp=hp;if(Cp>Cm,Cm=Cp;
C=log(Cp)/log(sqrt(D));print("D=",-D," h=",h," hp=",hp," C=",C)))}
D=-15         h=2      hp=2      C=0.511916049619630978775355357
D=-39         h=4      hp=4      C=0.756801438067480149325544162
D=-95         h=8      hp=8      C=0.913262080279460212705801846
D=-399        h=16     hp=16     C=0.925899677503555682939700450
D=-791        h=32     hp=32     C=1.038687593312750474942887870
D=-2519       h=64     hp=64     C=1.062075159346033035976072133
D=-10295      h=128    hp=128    C=1.050289653382181398975491576
D=-39431      h=256    hp=256    C=1.048009122470377471769618833
D=-132599     h=512    hp=512    C=1.057783767181715434360601717
D=-328319     h=1024   hp=1024   C=1.091420745999194423260975917
D=-1333631    h=2048   hp=2048   C=1.081244297733198664388474474
D=-4599839    h=4096   hp=4096   C=1.084346236368631648159879902
D=-18855359   h=8192   hp=8192   C=1.075781736259555689965062133
D=-63836951   h=16384  hp=16384  C=1.079918254667737276882538104
D=-266675639  h=32768  hp=32768  C=1.071791801714607295960939150
D=-966467519  h=65536  hp=65536  C=1.072093388756179498237226639
\end{verbatim}
\normalsize

We shall examine elsewhere all these strange phenomena which seems valid for 
all $p$ and suggest the existence of families for which the
$p$-part of the class number has maximal values, so that
$\ds\frac{{\rm log}(\order (\Cl_K \otimes \Z_p))}{{\rm log}(\sqrt{D_K}\,)} \to 1$
as $D_K \to \infty$.

\subsubsection{Reciprocal study}
To try to suggest the existence of analogous families giving huge $p$-class groups, 
a trick is to consider normic equations, in integers $a, b$, of the form:
$$a^2+m\,b^2=4 \cdot q^{p^\rho}, \ \, {\rm gcd}(a,b) \in \{1,2\}, $$ 

where $q>1$ is any fixed integer and $\rho$ an exponent as large as possible;
then when $a \geq 1$ increases, we deduce $b$ and the  square free integer $m$.
This kind of experiment has shown, in \cite[\S\,5.3]{Gr2}, that there exist huge 
discriminants giving interesting $p$-adic invariants. Moreover the function:
$$C_{K,p} := \frac{{\rm log}(\order (\Cl_K \otimes \Z_p))}{{\rm log}(\sqrt{D_K})}, $$

giving (for any $\varepsilon > 0$):
$${\rm log}(C_\varepsilon) \geq \big(C_{K,p} - \varepsilon \big)
\cdot{\rm log}(\sqrt{D_K}) , $$

may constitute an obstruction for the strong $\varepsilon$-conjecture 
as soon as:
$$\ds \liminf_{K \in {\mathcal K}} C_{K,p} > 0$$ 

for an infinite subfamily ${\mathcal K}$ of the set of imaginary quadratic fields.
But a priori, this does not affect the $p$-rank $\varepsilon$-conjecture.

\smallskip
Of course, the right member of the normic equation being rapidely too 
large when $\rho$ increases, the experimentation is very limited regarding 
PARI/GP possibilities. However, even for small values of $\rho$ predicting,
a priori, $p$-classes of order aroud $p^\rho$ we obtain much large orders, and
the following numerical results may be convincing enough about the infiniteness 
of such utmost examples. 

\footnotesize
\begin{verbatim}
{p=3;rho=4;q=2;Y=4*q^(p^rho);ba=1;Ba=sqrt(Y);H=1;for(a=ba,Ba,B=Y-a^2;
m=core(B);D=m;if(Mod(m,4)!=-1,D=4*m);b=component(core(B,1),2);
if(gcd(a,b)>2,next);h=qfbclassno(-D);vh=valuation(h,p);hp=p^vh;
if(hp>H,H=hp;Cp=log(hp)/log(sqrt(D));Hp=component(quadclassunit(-D),2);
d=component(matsize(Hp),2);L=List;for(k=1,d,c=component(Hp,k);
w=valuation(c,p);if(valuation(c,p)!=0,listput(L,p^w)));
print("D=",D," a=",a," b=",b," Cp=",Cp," hp=",hp," Hp=",L)))}
rho=4
D=9671406556917033397649407 a=1 b=1 Cp=0.152767 hp=81 Hp=[81]
D=197375644018714967298967  a=5 b=7 Cp=0.204814 hp=243 Hp=[243]
D=9671406556917033397648447 a=31 b=1 Cp=0.229151 hp=729 Hp=[729]
D=9671406556917033397648319 a=33 b=1 Cp=0.267343 hp=2187 Hp=[729,3]
D=9671406556917033397644647 a=69 b=1 Cp=0.305534 hp=6561 Hp=[243,27]
D=9671406556917033397435039 a=463 b=1 Cp=0.343726 hp=19683 Hp=[19683]
D=9671406556917033397373783 a=525 b=1 Cp=0.381918 hp=59049 Hp=[59049]
D=9671406556917033395993039 a=1287 b=1 Cp=0.420110 hp=177147 Hp=[177147]
D=9671406556917033372819119 a=4983 b=1 Cp=0.496494 hp=1594323 Hp=[531441,3]
D=9671406556917022018093783 a=106675 b=1 Cp=0.534686 hp=4782969
                                                             Hp=[1594323,3]
{p=2;rho=6;q=2;Y=4*q^(p^rho);ba=1;Ba=sqrt(Y);H=1;for(a=ba,Ba,B=Y-a^2;
m=core(B);D=m;if(Mod(m,4)!=-1,D=4*m);b=component(core(B,1),2);
if(gcd(a,b)>2,next);h=qfbclassno(-D);vh=valuation(h,p);hp=p^vh;
if(hp>H,H=hp;Cp=log(hp)/log(sqrt(D));Hp=component(quadclassunit(-D),2);
d=component(matsize(Hp),2);L=List;for(k=1,d,c=component(Hp,k);
w=valuation(c,p);if(valuation(c,p)!=0,listput(L,p^w)));
print("D=",D," a=",a," b=",b," Cp=",Cp," hp=",hp," Hp=",L)))}
rho=6
D=8198552921648689607 a=1 b=3 Cp=0.445646 hp=16384 Hp=[512,2,2,2,2,2]
D=18446744073709551615 a=2 b=2 Cp=0.468750 hp=32768 Hp=[512,4,2,2,2,2]
D=73786976294838206415 a=7 b=1 Cp=0.575757 hp=524288 Hp=[32768,4,2,2]
D=73786976294838175135 a=177 b=1 Cp=0.636363 hp=2097152 Hp=[8192,16,2,2,2,2]
D=73786976294831146815 a=2657 b=1 Cp=0.666666 hp=4194304 
                                                   Hp=[16384,8,2,2,2,2,2]
D=8198552921599834167 a=20969 b=3 Cp=0.732133 hp=8388608 
                                                   Hp=[65536,4,2,2,2,2,2]
D=73786976293564644495 a=35687 b=1 Cp=0.787878 hp=67108864 
                                                   Hp=[2048,512,2,2,2,2,2,2]
D=73786976290585731943 a=65211 b=1 Cp=0.969696 hp=4294967296 
                                                   Hp=[33554432,4,2,2,2,2,2]
                                                   
{p=2;rho=6;q=3;Y=4*q^(p^rho);ba=1;Ba=sqrt(Y);H=1;for(a=ba,Ba,B=Y-a^2;
m=core(B);D=m;if(Mod(m,4)!=-1,D=4*m);b=component(core(B,1),2);
if(gcd(a,b)>2,next);h=qfbclassno(-D);vh=valuation(h,p);hp=p^vh;
if(hp>H,H=hp;Cp=log(hp)/log(sqrt(D));Hp=component(quadclassunit(-D),2);
d=component(matsize(Hp),2);L=List;for(k=1,d,c=component(Hp,k);
w=valuation(c,p);if(valuation(c,p)!=0,listput(L,p^w)));
print("D=",D," a=",a," b=",b," Cp=",Cp," hp=",hp," Hp=",L)))}
rho=6
D=13734735281170049938631396357123 a=1 b=1 Cp=0.154682 hp=256 Hp=[64,2,2]
D=53651309692070507572778892020 a=2 b=32 Cp=0.272429 hp=8192 
                                               Hp=[64,4,4,2,2,2]
D=13734735281170049938631396357108 a=4 b=2 Cp=0.270694 hp=16384 
                                               Hp=[64,8,2,2,2,2,2]
D=549389411246801997545255854283 a=7 b=5 Cp=0.303662 hp=32768 
                                               Hp=[64,4,2,2,2,2,2,2,2]
D=38046358119584625868785031460 a=8 b=38 Cp=0.379179 hp=262144 
                                               Hp=[512,8,2,2,2,2,2,2]
D=13734735281170049938631396327195  a=173 b=1 Cp=0.464048 hp=16777216
                                               Hp=[32768,8,4,2,2,2,2]                             
\end{verbatim}
\normalsize

We note the exceptional case $D_K=73786976290585731943$ with 
$$\Cl_K \otimes \Z_2 \simeq \Z/2^{25}\Z \times  \Z/2^{2}\Z \times  (\Z/2\Z)^5 ,$$ 
giving the large value $C_{K,2}=0.969696$.

\smallskip
One computes that the group ${\mathcal T}_K$ that we shall study in the 
next section is isomorphic to $\Z/2^{2}\Z \times  (\Z/2\Z)^6$ which yields 
$[\wt K \cap H_K : K] = 2^{25}$, where $\wt K$ is the compositum 
of the $\Z_2$-extensions of $K$.

\smallskip
Many families are described by means of parametized radicals
as the family of fields $K=\Q(\sqrt{k^2-q^n})$ with any prime $q \ne 2$, 
$k^2-q^n<0$, in which $\Cl_K$ has, under some conditions on the parameters, 
an element of order $n$ (see \cite[Theorem 3.1]{BH} and its bibliography); 
applied to $n=p^r$, we get, for $C_{K,p} $, the upper bound
$\frac{O(1) \,r}{p^r} \to 0$ as $r \to \infty$,
not sufficient to give ``bad families''.

\smallskip
It is difficult to say if some of the above huge discriminants may be obtained with
explicit parametrized expressions.

\subsection{Reminders on $p$-ramification theory}
We intend to give now some analogies with the torsion group 
${\mathcal T}_K$ of the Galois group of the maximal abelian 
$p$-ramified (i.e., unramified outside $p$ and $\infty$) pro-$p$-extension of $K$; 
this extension contains the $p$-Hilbert class field of $K$ (in the ordinary 
sense) and the compositum of the $\Z_p$-extensions of $K$. This Galois group 
introduces the ``normalized'' $p$-adic regulator of $K$ defined in \cite[\S\,5]{Gr6}.

\smallskip
Since in this section the non-$p$-part of the class group does not
intervene, unless otherwise stated, we shall put, by abuse of notation, 
$\Cl_K := \Cl_K \otimes \Z_p$.

\subsubsection{Structure of the $p$-torsion group ${\mathcal T}_K$}
Let $K$ be any number field and let
$p\geq 2$ be a prime number; we denote by ${\mathfrak p} \mid p$ 
the prime ideals of $K$ dividing $p$. 
Consider the group $E_K$ of $p$-principal global units of $K$ (i.e., 
units $\varepsilon \equiv 1 \! \pmod{ \prod_{{\mathfrak p} \mid p} {\mathfrak p}}$).
For each ${\mathfrak p} \mid p$, let $K_{\mathfrak p}$ be the ${\mathfrak p}$-completion
of $K$ and $\ov {\mathfrak p}$ the corresponding prime ideal of the ring of integers 
of $K_{\mathfrak p}$; then let:

\smallskip
\centerline{$U_K := \Big \{u \in \plus_{{\mathfrak p}\, \mid\, p}K_{\mathfrak p}^\times, \ \,
u = 1+x, \  x \in \plus_{{\mathfrak p} \,\mid\, p} \ov {\mathfrak p} \Big\}\, \ \ \& \ \ \,
W_K := {\rm tor}_{\Z_ p}^{}(U_K)$,} 

\smallskip\noindent
the $\Z_ p$-module of principal local units at $p$ and its  torsion subgroup.

\smallskip
We consider the diagonal embedding $E_K \otimes \Z_ p \too U_K$ 
whose image is $\ov E_K$, the topological closure of $E_K$ in $U_K$. 

\smallskip
We assume in this paper that $K$ satisfies the Leopoldt conjecture at $p$. Whence
the following $p$-adic result (\cite[Lemma 3.1, Corollary 3.2]{Gr6},
\cite[Lemma III.4.2.4]{Gr1}, \cite[D\'efinition 2.11, Proposition 2.12]{J1}):

\begin{lemma} \label{exact}
Let $\mu_K^{}$ be the group of global roots of unity of $p$-power order of $K$. 
Under the Leopoldt conjecture for $p$ in $K$, we have 
${\rm tor}_{\Z_ p}^{}(\ov E_K) = \mu_K^{}$ and the exact sequence (where
${\rm log}$ is the $p$-adic logarithm):
$$1 \to W_K \big / \mu_K^{}  \tooo
 {\rm tor}_{\Z_ p}^{} \big(U_K \big / \ov E_K \big) 
 \mathop {\tooo}^{ \!\!{\rm log}}  {\rm tor}_{\Z_ p}^{}\big({\rm log}\big 
(U_K \big) \big / {\rm log} (\ov E_K) \big) \to 0. $$
\end{lemma}

Put ${\mathcal W}_K:= W_K /\mu_K^{} \  $ \& $ \ {\mathcal R}_K := 
{\rm tor}^{}_{\Z_ p} \big ({\rm log} (U_K) / {\rm log} (\ov E_K)\big)$.
Then the above exact sequence becomes
$1 \to {\mathcal W}_K  \tooo
 {\rm tor}_{\Z_ p}^{} \big(U_K \big / \ov E_K \big) 
\ds \mathop {\tooo}^{\!\!{\rm log}} {\mathcal R}_K \to 0$.

\smallskip
Let $\widetilde K$ be the compositum of the $\Z_ p$-extensions, $H_K$ 
the $p$-Hilbert class field and $H_K^{\rm pr}$ the maximal Abelian $p$-ramified 
pro-$p$-extension, of $K$. 
Then let $H_K^{\rm bp}$ be the Bertrandias--Payan field
(compositum of the $p$-cyclic extensions of $K$ embeddable in 
$p$-cyclic extensions of arbitrary large degree). 

\smallskip
In the following diagram, class field theory yields:
$$\hbox{${\rm Gal}(H_K^{\rm pr} / H_K) \simeq U_K/\ov E_K$
and ${\rm Gal}(H_K^{\rm pr}/H_K^{\rm bp}) \simeq {\mathcal W}_K$.} $$

We denote by $\wt \Cl_K$ the subgroup of the $p$-class group $\Cl_K$
corresponding to ${\rm Gal}(H_K /\wt K \cap H_K)$  by class field theory.

\smallskip
Then ${\mathcal R}_K$ is isomorphic to 
${\rm Gal}(H_K^{\rm bp} / \widetilde KH_K)$:
\unitlength=0.85cm 
$$\vbox{\hbox{\hspace{-2.8cm} 
 \begin{picture}(11.5,5.6)
\put(6.7,4.50){\line(1,0){1.3}}
\put(8.75,4.50){\line(1,0){2.0}}
\put(3.85,4.50){\line(1,0){1.4}}
\put(9.1,4.05){\footnotesize$\simeq\! {\mathcal W}_K$}
\put(4.3,2.50){\line(1,0){1.25}}
\bezier{350}(3.8,4.8)(7.6,5.8)(11.0,4.8)
\put(7.2,5.45){\footnotesize${\mathcal T}_K$}
\put(3.50,2.9){\line(0,1){1.25}}
\put(3.50,0.8){\line(0,1){1.4}}
\put(5.7,2.9){\line(0,1){1.25}}
\bezier{300}(3.9,0.55)(4.9,0.8)(5.6,2.3)
\put(5.2,1.3){\footnotesize$\simeq \! \Cl_K$}
\put(4.1,4.05){\footnotesize$\simeq\! \wt \Cl_K$}
\bezier{300}(6.3,2.5)(8.5,2.6)(10.8,4.3)
\put(8.0,2.6){\footnotesize$\simeq \! U_K/\ov E_K$}
\put(10.85,4.4){$H_K^{\rm pr}$}
\put(5.4,4.4){$\wt K H_K$}
\put(8.0,4.4){$H_K^{\rm bp}$}
\put(6.8,4.05){\footnotesize$\simeq\! {\mathcal R}_K$}
\put(3.3,4.4){$\wt K$}
\put(5.55,2.4){$H_K$}
\put(2.65,2.4){$\wt K \!\cap \! H_K$}
\put(3.4,0.38){$K$}
\end{picture}   }} $$
\unitlength=1.0cm

\smallskip
 We have $\order {\mathcal T}_K = \order \wt \Cl_K \cdot \order {\mathcal R}_K
 \cdot \order {\mathcal W}_K$ and the following inequalities:
\begin{equation} \label{rank1}
\begin{aligned}
{\rm rk}_p({\mathcal T}_K)& \leq {\rm rk}_p(\wt \Cl_K)+
{\rm rk}_p({\mathcal R}_K)+{\rm rk}_p({\mathcal W}_K) \\
& \leq {\rm rk}_p(\Cl_K) + r_1+r_2-1 + \order S_K, 
\end{aligned}
\end{equation}

where $(r_1, r_2)$ is the signature of $K$ and $S_K$ the
set of $p$-places of $K$. So, for a constant degree $d$,
the $p$-rank $\varepsilon$-conjecture for the $p$-class groups 
implies the $p$-rank $\varepsilon$-conjecture for the torsion 
groups ${\mathcal T}_K$ and conversely since we have the other inequality:
\begin{equation}
\begin{aligned} \label{rank2}
{\rm rk}_p(\Cl_K) & \leq {\rm rk}_p(\wt \Cl_K)+ 
{\rm rk}_p({\rm Gal}(\wt K \cap H_K/K)) \\
&\leq {\rm rk}_p(\wt \Cl_K)+ r_2+1 \leq {\rm rk}_p({\mathcal T}_K) + r_2+1.
\end{aligned}
\end{equation}

For more precise rank formulas for ${\mathcal T}_K$, 
see \cite[Corollary III.4.2.3]{Gr1} and the reflection theorem
that we  shall recall in \S\,\ref{refl}.

\subsubsection{The $p$-adic Brauer--Siegel conjecture for ${\mathcal T}_K$}
We have proposed in \cite{Gr2}, for the totally real case, after 
extensive numerical computations, the following conjecture:

\begin{conjecture} \label{conj1}
Let $p \geq 2$ be prime and let $d$ be a given degree.
For any number field $K$ (under Leopoldt's conjecture), let 
${\mathcal T}_K$ be the torsion group of the Galois group of 
the maximal abelian $p$-ramified pro-$p$-extension of $K$.
There exists a constant $\wt{\mathcal C}_{d,p}$ such that:
\begin{equation*}
\order {\mathcal T}_K \leq (\sqrt{D_K}\,)^{\wt{\mathcal C}_{d,p}}, \  
\hbox{for all $K$ totally real of degree $d$}.
\end{equation*}
\end{conjecture}

We put, for $p$ fixed and for any totally real number field $K$:
\begin{equation}\label{CKp}
\wt{\mathcal C}_{K,p} := \ds\frac{{\rm log}(\order {\mathcal T}_K)}
{{\rm log}(\sqrt{D_K})} \leq \wt{\mathcal C}_{d,p}. 
\end{equation}

In practice, $\wt{\mathcal C}_{K,p}$ may be much smaller
than $1$ (and it is often $0$), except very sparse cases as that of $K=\Q(\sqrt{19})$ 
and $p=13599893$, for which ${\mathcal T}_K = {\mathcal R}_K \simeq \Z/p\Z$, 
whence $\wt{\mathcal C}_{K,p} =\ds \frac{{\rm log}(\order {\mathcal T}_K)}
{{\rm log}(\sqrt {4\times19})}=\frac{{\rm log}(13599893)}
{{\rm log}(\sqrt {4\times19})} = 7.5855$. But $\Cl_K=1$.

\subsubsection{Estimation of $\wt{\mathcal C}_{F_{N,p},p}$}
Put $F :=F_{N,p}$ and $\wt{\mathcal C}_{F_N,p} =: \wt{\mathcal C}_F$
 for $p$ fixed. 
Then, from the computations in Subsection \ref{estimation}, 
using for  ${\mathcal T}_F$ the analog of Chevalley's 
formula given in \cite[Theorem IV.3.3]{Gr1}, we can put similarly
$\order {\mathcal T}_F =: p^{N-r+\wt\Delta(N)}$, $\wt\Delta(N) \geq 0$, 
where $r \geq 0$ depends on $p$-adic properties of the ramified 
primes $\ell \ne p$, and we may estimate that, as $N\to\infty$:
$$\wt{\mathcal C}_{F} \approx \frac{(N-r+ \wt\Delta(N))\,{\rm log}(p)}
{\frac{p-1}{2}\, \big[ N {\rm log}(N) + N\,\gamma_p + 
\frac{1}{2}\,{\rm log}(N) +O(1) \big]} \sim c_p \cdot
\frac{1+o(1)}{{\rm log}(N)}, $$

where $c_p = \frac{2\,{\rm log}(p)}{p-1}$ 
and assuming a small order of magnitude of $\wt\Delta(N)$.

\smallskip
Give a program computing (in ${\sf Cp}$)
$\wt{\mathcal C}_{F}$; in the
imaginary quadratic case for $p=2$, the conjectural inequality implies
$\ds \order (\wt \Cl_F) \leq  (\sqrt{D_F}\,)^{\wt{\mathcal C}_{2,2}}$
(indeed, $\order{\mathcal R}_F=1$ and $\order {\mathcal T}_F =
\order\wt \Cl_F \cdot {\mathcal W}_F$, where ${\mathcal W}_F=2$ 
(resp. $1$) if $m \equiv \pm1 \pmod 8$ (resp. if not)). So $\Cl_F$
(in ${\sf Clres}$) may be larger than $\order {\mathcal T}_F$ (in ${\sf Tor}$):

\footnotesize
\begin{verbatim}
{p=2;n=12;m=1;for(N=2,100,el=prime(N);m=(-1)^((el-1)/2)*el*m;P=x^2-m;
K=bnfinit(P,1);D=abs(m);Kpn=bnrinit(K,p^n);r=1;if(m<0,r=2);L=List;
Hpn=component(component(Kpn,5),2);e=component(matsize(Hpn),2);T=1;
for(k=1,e-r,c=component(Hpn,e-k+1);if(Mod(c,p)==0,q=p^valuation(c,p);
T=T*q;listinsert(L,q,1)));C8=component(K,8);C81=component(C8,1);
h=component(C81,1);Clord=component(C81,2);K=bnfnarrow(K);
Cl=component(K,2);print("m=",m," Clres=",Clres," Clord =",Clord);
print("Structure of T=",L);print("#Tor=",T," Cp=",log(T)/log(sqrt(D))))}

m=-15   Clres=[2]
Structure of Tor=[2]
#Tor=2              Cp=0.51191604961963097877535535772960454081
m=105   Clres=[2,2]=[2]xClord 
Structure of Tor=[2,2]
#Tor=4              Cp=0.59574824743531323067786608868687642325
m=-1155   Clres=[2,2,2]
Structure of Tor=[2,2,2]
#Tor=8              Cp=0.58975726471501581115878339498474155345
m=-15015   Clres=[12,2,2,2]
Structure of Tor=[2,2,2,2]
#Tor=16             Cp=0.57661327808675875001115538902772596330
m=-255255   Clres=[16,2,2,2,2]
Structure of Tor=[2,2,2,2,2]
#Tor=32             Cp=0.55674390390043840097934284424073618196
m=4849845   Clres=[4,2,2,2,2,2]=[2]xClord
Structure of Tor=[4,2,2,2,2,2]
#Tor=128            Cp=0.63036067699149527703810495075838580918
m=-111546435   Clres=[42,2,2,2,2,2,2]
Structure of Tor=[2,2,2,2,2,2,2]
#Tor=128            Cp=0.52369594802182316940823598390366865409
m=-3234846615   Clres=[308,2,2,2,2,2,2,2]
Structure of Tor=[4,2,2,2,2,2,2,2]
#Tor=512            Cp=0.56978162829823280646049502887398109011
m=100280245065   Clres=[2,2,2,2,2,2,2,2,2]=[2]xClord
Structure of Tor=[2,2,2,2,2,2,2,2,2]
#Tor=512            Cp=0.49254011775311327297187105815891979776
m=3710369067405   Clres=[34,2,2,2,2,2,2,2,2,2]=[2]xClord
Structure of Tor=[2,2,2,2,2,2,2,2,2,2]
#Tor=1024           Cp=0.47898799604336105240349464751069989167
m=152125131763605   Clres=[2,2,2,2,2,2,2,2,2,2,2]=[2]xClord
Structure of Tor=[4,2,2,2,2,2,2,2,2,2,2]
#Tor=4096           Cp=0.50942162732997733185991268818214459006
m=-6541380665835015   Clres=[28284,2,2,2,2,2,2,2,2,2,2,2]
Structure of Tor=[2,2,2,2,2,2,2,2,2,2,2,2]
#Tor=4096           Cp=0.45680772041407406738242757697825579415
m=307444891294245705   Clres=[14,2,2,2,2,2,2,2,2,2,2,2,2]=[2]xClord
Structure of Tor=[2,2,2,2,2,2,2,2,2,2,2,2,2]
#Tor=8192           Cp=0.44755741322126195443312989366082458089
m=16294579238595022365   Clres=[2,2,2,2,2,2,2,2,2,2,2,2,2,2]=[2]xClord
Structure of Tor=[32,2,2,2,2,2,2,2,2,2,2,2,2,2]
#Tor=262144         Cp=0.56407742575833976013164891102271736655
m=-961380175077106319535   Clres=[1210734,2,2,2,2,2,2,2,2,2,2,2,2,2,2]
Structure of Tor=[2,2,2,2,2,2,2,2,2,2,2,2,2,2,2]
#Tor=32768          Cp=0.43039341330765014032907101774855577316
m=-58644190679703485491635   Clres=[1622526,2,2,2,2,2,2,2,2,2,2,2,2,2,2,2]
Structure of Tor=[2,2,2,2,2,2,2,2,2,2,2,2,2,2,2,2]
#Tor=65536          Cp=0.42308787191112227068703992014107942867
\end{verbatim}
\normalsize

We do not know an algorithm computing 
the ``exceptional elements'' of ${\mathcal T}_F$ as for $p$-class groups.

\smallskip
The case $p=3$ is similar and gives for instance for the $16$ cyclic cubic
fields of conductor $f=7 \cdot 13 \cdot 19 \cdot 31 \cdot 37$:

\footnotesize
\begin{verbatim}
P=x^3+x^2-661054*x+49725976       Cl=[3,3,3,3,3]
Structure of Tor=[9,3,3,3,3]
#Tor=729   Cp=0.45459180523024141673723101157712338880
P=x^3+x^2-661054*x+198463201      Cl=[6,6,3,3]
Structure of Tor=[9,3,3,3,3]
#Tor=729   Cp=0.45459180523024141673723101157712338880
P=x^3+x^2-661054*x+97321888       Cl=[3,3,3,3]
Structure of Tor=[3,3,3,3,3]
#Tor=243   Cp=0.37882650435853451394769250964760282400
P=x^3+x^2-661054*x-186270421      Cl=[39,3,3,3,3]
Structure of Tor=[9,3,3,3,3]
#Tor=729   Cp=0.45459180523024141673723101157712338880
P=x^3+x^2-661054*x-79179619       Cl=[3,3,3,3]
Structure of Tor=[3,3,3,3]
#Tor=81    Cp=0.30306120348682761115815400771808225920
P=x^3+x^2-661054*x-188253584      Cl=[6,6,3,3]
Structure of Tor=[9,3,3,3,3]
#Tor=729   Cp=0.45459180523024141673723101157712338880
P=x^3+x^2-661054*x+138968311      Cl=[3,3,3,3,3]
Structure of Tor=[3,3,3,3]
#Tor=81    Cp=0.30306120348682761115815400771808225920
P=x^3+x^2-661054*x-146607161      Cl=[3,3,3,3]
Structure of Tor=[3,3,3,3,3]
#Tor=243   Cp=0.37882650435853451394769250964760282400
P=x^3+x^2-661054*x-158506139      Cl=[3,3,3,3,3]
Structure of Tor=[3,3,3,3,3]
#Tor=243   Cp=0.37882650435853451394769250964760282400
P=x^3+x^2-661054*x-140657672      Cl=[3,3,3,3]
Structure of Tor=[3,3,3,3,3]
#Tor=243   Cp=0.37882650435853451394769250964760282400
P=x^3+x^2-661054*x+186564223      Cl=[6,6,3,3]
Structure of Tor=[3,3,3,3]
#Tor=81    Cp=0.30306120348682761115815400771808225920
P=x^3+x^2-661054*x+81456584       Cl=[3,3,3,3]
Structure of Tor=[3,3,3,3,3]
#Tor=243   Cp=0.37882650435853451394769250964760282400
P=x^3+x^2-661054*x+206395853      Cl=[3,3,3,3]
Structure of Tor=[3,3,3,3]
#Tor=81    Cp=0.30306120348682761115815400771808225920
P=x^3+x^2-661054*x-206102051      Cl=[3,3,3,3]
Structure of Tor=[3,3,3,3]
#Tor=81    Cp=0.30306120348682761115815400771808225920
P=x^3+x^2-661054*x+2130064        Cl=[12,12,3,3]
Structure of Tor=[3,3,3,3,3,3]
#Tor=729   Cp=0.45459180523024141673723101157712338880
P=x^3+x^2-661054*x+27911183       Cl=[3,3,3,3,3]
Structure of Tor=[3,3,3,3,3]
#Tor=243   Cp=0.37882650435853451394769250964760282400
\end{verbatim}
\normalsize

\smallskip
Now give some examples for $p=2$ and $p=3$ where the constant 
$\wt{\mathcal C}_{F}$ is rather large; the degree $p$ cyclic extensions 
$F$ are not necessary of the form $F_N$.

\smallskip
(i) $p=2$, $m$ divides $5 \cdot 7 \cdots 41 \cdot 43$.

\footnotesize
\begin{verbatim}
m=5005 Clres=[2,2,2] Clord =[2,2]
Structure of Tor=[4,2,2]
#Tor=16      Cp=0.65098051255057821664489247470558390764
m=1078282205 Clres=[4,2,2,2,2,2,2] Clord =[4,2,2,2,2,2]
Structure of Tor=[32,4,2,2,2,2,2]
#Tor=4096    Cp=0.79983769534579979852611700457236749058
m=215656441 Clres=[2,2,2,2,2,2] Clord =[2,2,2,2,2]
Structure of Tor=[8,2,2,2,2,2]
#Tor=256     Cp=0.57794783195778051534152523962133180636
m=436092044389001 Clres=[2,2,2,2,2,2,2,2,2,2] Clord =[2,2,2,2,2,2,2,2,2]
Structure of Tor=[32,2,2,2,2,2,2,2,2,2]
#Tor=16384   Cp=0.57575701869336876560350400103707496594
m=46189 Clres=[2,2,2] Clord =[2,2]
Structure of Tor=[16,2,2]
#Tor=64      Cp=0.77443028965455279095387556052161745644
m=221 Clres=[4] Clord =[2]
Structure of Tor=[16]
#Tor=16      Cp=1.0272342185833848333397010211662592994
m=435656388001 Clres=[2,2,2,2,2,2,2] Clord =[2,2,2,2,2,2]
Structure of Tor=[8,2,2,2,2,2,2]
#Tor=512     Cp=0.46554453503678235229309396294036417544
\end{verbatim}
\normalsize

\smallskip
We note the case of $m=221$ for which $\Cl_F^{\rm res} \simeq \Z/4\Z$,
$\Cl_F^{\rm ord} \simeq \Z/2\Z$, but ${\mathcal T}_F \simeq \Z/16\Z$
due to the $2$-adic regulator. This gives  the exceptional value
$\wt{\mathcal C}_{F,2} \approx 1.02723422$.

\smallskip
(ii) $p=3$, $f$ divides $13 \cdot 19 \cdot 31 \cdot 37 \cdot 43$.

\footnotesize
\begin{verbatim}
f=10621   P=x^3+x^2-3540*x-60579    Cl=[3,3]
Structure of Tor=[9,9,3]
#Tor=243     Cp=0.33667761382504192691963484073748684274
P=x^3+x^2-4060762*x-3150249179    Cl=[3,3,3,3]
Structure of Tor=[27,27,3,3,3]
#Tor=19683   Cp=0.60601970488507546845534271332747631693
f=12182287   P=x^3+x^2-4060762*x+187697459    Cl=[9,9,3,3,3]
Structure of Tor=[9,9,3,3,3]
#Tor=2187    Cp=0.47134865935505869768748877703248157983
f=12182287   P=x^3+x^2-4060762*x+2380509119    Cl=[3,3,3,3]
Structure of Tor=[9,3,3,3,3]
#Tor=729     Cp=0.40401313659005031230356180888498421129
f=12182287   P=x^3+x^2-4060762*x-2309671376    Cl=[3,3,3,3,3]
Structure of Tor=[9,9,3,3,3]
#Tor=2187    Cp=0.47134865935505869768748877703248157983
f=641173   P=x^3-x^2-213724*x-29968901    Cl=[6,6,3]
Structure of Tor=[9,3,3,3]
#Tor=243     Cp=0.33667761382504192691963484073748684274
f=392977   P=x^3-x^2-130992*x+6826156    Cl=[3,3,3]
Structure of Tor=[9,3,3,3]
#Tor=243     Cp=0.33667761382504192691963484073748684274
f=7657   P=x^3+x^2-2552*x+47360    Cl=[9,3,3]
Structure of Tor=[9,9,3]
#Tor=243     Cp=0.33667761382504192691963484073748684274
\end{verbatim}
\normalsize

\subsection{$p$-adic Brauer--Siegel conjecture versus $\varepsilon$-conjectures}

The $p$-adic Brauer--Siegel Conjecture \ref{conj1} concerns more
essentially the totally real case for the following reasons which are
yet given by the rank inequalities \eqref{rank1} and \eqref{rank2}.

\subsubsection{Analysis by means of CM-fields}
We have, with obvious notations and $p \ne 2$:
\begin{equation*}
\begin{aligned}
& \Cl_K = \Cl_K^-  \plus \Cl_K^+ , 
\hspace{0.8cm}{\mathcal T}_K = {\mathcal T}_K^- \plus {\mathcal T}_K^+, \\
& {\mathcal R}_K = {\mathcal R}_K^- \plus {\mathcal R}_K^+,
\hspace{0.6cm} {\mathcal W}_K = {\mathcal W}_K^- \plus {\mathcal W}_K^+; 
\end{aligned}
\end{equation*}

but we have the following properties which explain the differences 
between real fields and non real ones (under the Leopoldt conjecture):

\medskip
(i) $\order {\mathcal R}_K^- = 1$ since all the units 
of infinite order of $K$ are real;

\smallskip
(ii) $\order \Cl_K^- = \order \wt \Cl_K^{{}_-} \cdot \order 
{\rm Gal}(\wt K \cap H_K/K)^-$ where 
${\rm Gal}(\wt K \cap H_K/K)^- \simeq \Cl_K^-/ \wt \Cl_K^{{}_-}$
may be large (but with bounded rank) since 
${\rm Gal}(\wt K/K)^- \simeq \Z_p^{\frac{d}{2}}$ contrary to 
${\rm Gal}(\wt K/K)^+ \simeq \Z_p$;

\smallskip
(iii) $\order {\mathcal T}_K^- = \order \wt \Cl_K^{{}_-} \cdot \order {\mathcal W}_K^-$ 
is essentially equal to $\order \wt \Cl_K^{{}_-}$ since ${\mathcal W}_K^-$
is most often trivial and does not intervene in estimations of class 
groups for $d$ fixed;

\smallskip
(iv) $\order {\mathcal R}_K^+$ is the main $p$-adic invariant which may
be nontrivial for much primes $p$, even if we have conjectured in \cite{Gr9} 
that it is trivial for all $p$ large enough;

\smallskip
(v) $\order \Cl_K^+$ is essentially equal to $\order \wt \Cl^{{}_+}_K$
since the part of the Hilbert class field, contained in the cyclotomic
$\Z_p$-extension, is very limited;

\smallskip
(vi) $\order {\mathcal T}_K^+ =  \order \wt \Cl_K^{{}_+} \cdot 
\order {\mathcal R}_K^+ \cdot \order {\mathcal W}_K^+$
is thus essentially the product $\order \Cl_K^+ \cdot \order {\mathcal R}_K^+$.

\medskip
So if we assume that $\order {\mathcal T}_K^+$ is controled, we may
consider that $\order \Cl_K^+$ is much less than $\order {\mathcal T}_K^+$
because of the regulator; then, since $\order {\mathcal T}_K^-$ is 
independent of any regulator, it is mesured by a divisor of $\order \Cl_K^-$ 
(equal to $\order \wt \Cl_K^{{}_-}$) and by $\order {\mathcal W}_K^-$ controled, 
which explains (from the factor $[\wt K \cap H_K : K]$) a bigger order of magnitude 
of $\order \Cl_K^-$ regarding $\order {\mathcal T}_K^-$ than 
$\order \Cl_K^+$ regarding $\order {\mathcal T}_K^+$.

\subsubsection{Computation of $\wt{\mathcal C}_{K,p}$ for imaginary quadratic fields}

We shall illustrate the cases $p=2$, then $p=3$, in various intervals of negative 
discriminants to observe the local decreasing of the variable 
$\wt{\mathcal C}_{K,p}$ \eqref{CKp}.

\smallskip
Note that for $p>3$, the group ${\mathcal W}_K$ is trivial contrary to the cases
$p=2$ and $3$ where ${\mathcal W}_K$ may be $\Z/p\Z$, which must probably be
discarded in our considerations.

\medskip
{\bf (a) Case $p=2$}.
The case $p=2$ is interesting because of the influence of exceptional classes
and gives ($v_p(\order {\mathcal T}_K)$ in ${\sf vptor}$, $\wt{\mathcal C}_{K,p}$
in ${\sf Cp}$):

\smallskip
\footnotesize
\begin{verbatim}
{p=2;bD=10^8;BD=2*10^8;Lp=log(p);vp=0;n=20;
for(D=bD,BD,e=valuation(D,2);M=D/2^e;if(core(M)!=M,next);
if((e==1 || e>3)||(e==0 & Mod(M,4)!=-1)||(e==2 & Mod(M,4)==-1),next);
m=D;if(e!=0,m=D/4);P=x^2+m;K=bnfinit(P,1);Kpn=bnrinit(K,p^n);
C5=component(Kpn,5);Hpn0=component(C5,1);Hpn=component(C5,2);
Hpn1=component(Hpn,1);vptor=valuation(Hpn0/Hpn1,p)-(n-1);
if(vptor>vp,vp=vptor);if(vptor>=vp,Cp=vptor*Lp/log(sqrt(D));
print("D=",-D," m=",-m," vptor=",vptor," Cp=",Cp)))}

p=2, Interval [10^6, 2*10^6]
D=-1000011 m=-1000011    vptor=3   Cp=0.301029755983435929933445793
D=-1000020 m=-250005     vptor=3   Cp=0.3010295598834164958938994188
D=-1000036 m=-250009     vptor=4   Cp=0.4013722816881976053812061427
D=-1000132 m=-250033     vptor=5   Cp=0.5017118661610285687682449315
(...)
D=-1003620 m=-250905     vptor=5   Cp=0.5015854691511746133432777519
D=-1003940 m=-250985     vptor=6   Cp=0.6018886779424325532238667109
D=-1005843 m=-1005843    vptor=7   Cp=0.702107244840955486811297669
D=-1007492 m=-251873     vptor=8   Cp=0.8023131911871206028276870051
(...)
D=-1327972 m=-331993     vptor=8   Cp=0.7865966565831969109249007496
D=-1345476 m=-336369     vptor=9   Cp=0.8841001125437591214944446738
D=-1347524 m=-336881     vptor=10  Cp=0.982227596578129040877631145

p=2, Interval [10^7, 2*10^7]
D=-10000004 m=-2500001   vptor=3   Cp=0.25802570416574861895099915
(...)
D=-10000136 m=-2500034   vptor=3   Cp=0.25802549285588586651933610
D=-10000212 m=-2500053   vptor=4   Cp=0.34403382825873844184039068
D=-10000228 m=-2500057   vptor=5   Cp=0.43004224263528166242354902
(...)
D=-10001220 m=-2500305   vptor=5   Cp=0.43003959612044524303603555
D=-10001355 m=-10001355  vptor=7   Cp=0.6020549303754533561270683
(...)
D=-10028164 m=-2507041   vptor=10  Cp=0.8599356519990630566433445
D=-11423624 m=-2855906   vptor=10  Cp=0.8530415407428446759627785
D=-11434244 m=-2858561   vptor=11  Cp=0.9382920445879771130663980
D=-19227908 m=-4806977   vptor=11  Cp=0.9092149504336010969244116

p=2, Interval [10^8, 2*10^8]
D=-100000011 m=-100000011 vptor=2   Cp=0.15051499693318294862950
D=-100000020 m=-25000005  vptor=3   Cp=0.225772494296692430574969
(...)
D=-100000072 m=-25000018  vptor=3   Cp=0.225772487923331962924891
D=-100000120 m=-25000030  vptor=4   Cp=0.301029976053644338278634
(...)
D=-100000228 m=-25000057  vptor=4   Cp=0.301029958404363471726365
D=-100000324 m=-25000081  vptor=6   Cp=0.451544914074197327895496
(...)
D=-100009811 m=-100009811  vptor=6  Cp=0.45154258866246136591601
D=-100009988 m=-25002497   vptor=7  Cp=0.526799636159210448373252
(...)
D=-100042692 m=-25010673   vptor=9  Cp=0.677301796362621931199437
D=-100120215 m=-100120215  vptor=11 Cp=0.8277784989602807429303
D=-100703939 m=-100703939  vptor=11 Cp=0.8275173634473368234393
D=-101091716 m=-25272929   vptor=13 Cp=0.97777114254342282551717
D=-196241540 m=-49060385   vptor=13 Cp=0.94380528108729550144090
\end{verbatim}
\normalsize

\smallskip
One sees some influence of genus theory since, for $D=-101091716$, 
we have $\wt{\mathcal C}_{K,2} \approx 0.977771$ because of 
$\order {\mathcal T}_K = 2^{13}$, but to be put in relation with 
$\wt{\mathcal C}_{K,2} \approx 0.982227$ for $D=-1347524$, of the
first interval, with $\order {\mathcal T}_K = 2^{10}$. 

\smallskip
The structure of the class group given by PARI/GP is ${\sf [1024, 2, 2, 2]}$.

\smallskip
Then consider the program computing the structure of ${\mathcal T}_K$
\cite[Programme I, \S\,3.2]{Gr8} that we recall for the convenience of the reader 
(choose ${\sf p}$, ${\sf nt}$ such that ${\sf p^{\rm nt}}$ be a multiple of the 
exponent of ${\mathcal T}_K$, then the polynomial ${\sf P}$):

\smallskip
\footnotesize
\begin{verbatim}
{p=2;nt=32;P=x^2+101091716;K=bnfinit(P,1);Kpn=bnrinit(K,p^nt);
S=component(component(Kpn,1),7);r=component(component(S,2),2)+1;
Hpn=component(component(Kpn,5),2);L=List;e=component(matsize(Hpn),2);
R=0;for(k=1,e-r,c=component(Hpn,e-k+1);if(Mod(c,p)==0,R=R+1;
listinsert(L,p^valuation(c,p),1)));print("Structure of T: ",L)}
\end{verbatim}
\normalsize

Then we obtain ${\mathcal T}_K \simeq \Z/2^{10}\Z \times \Z/2^2\Z \times \Z/2\Z$. 
The difference comes from ${\mathcal W}_K \simeq \Z/2\Z$ 
(since $-101091716 \equiv -1 \pmod {16}$) and $[\wt K\cap H_K : K]=2$.

\smallskip
{\bf (b) Case $p=3$}.

\footnotesize
\begin{verbatim}
{p=3;bD=10^6;BD=2*10^6;Lp=log(p);vp=0;n=8;
for(D=bD,BD,e=valuation(D,2);M=D/2^e;if(core(M)!=M,next);
if((e==1 || e>3)||(e==0 & Mod(M,4)!=-1)||(e==2 & Mod(M,4)==-1),next);
m=D;if(e!=0,m=D/4);P=x^2+m;K=bnfinit(P,1);Kpn=bnrinit(K,p^n);
C5=component(Kpn,5);Hpn0=component(C5,1);Hpn=component(C5,2);
Hpn1=component(Hpn,1);vptor=valuation(Hpn0/Hpn1,p)-(n-1);
if(vptor>vp,vp=vptor);if(vptor>=vp,Cp=vptor*Lp/log(sqrt(D));
print("D=",-D," m=",-m," vptor=",vptor," Cp=",Cp)))}
p=3, Interval [10^6, 2*10^6]
D=-1000011 m=-1000011     vptor=1   Cp=0.1590402916116620131420348520
D=-1000020 m=-250005      vptor=1   Cp=0.1590401880079362819278196125
D=-1000043 m=-1000043     vptor=1   Cp=0.1590399232477087416304663599
(...)
D=-1020548 m=-255137      vptor=3   Cp=0.4764198506806243371783107010
D=-1021332 m=-255333      vptor=4   Cp=0.6351912130796763789096457580
D=-1022687 m=-1022687     vptor=5   Cp=0.7939129439088033794338891302
(...)
D=-1898859 m=-1898859     vptor=5   Cp=0.7599296140003213455753306574
p=3, Interval [10^7, 10^7+10^6]
D=-10002927 m=-10002927   vptor=3   Cp=0.40895365007950900178024254
D=-10003224 m=-2500806    vptor=4   Cp=0.54527052902405852884991426
(...)
D=-10065279 m=-10065279   vptor=4   Cp=0.54506139909286008239956811
D=-10066440 m=-2516610    vptor=5   Cp=0.68132187532448910070906763
(...)
D=-14316744 m=-3579186    vptor=5   Cp=0.66675746065456780942779488
D=-14547531 m=-14547531   vptor=6   Cp=0.79933316809564910995167969
(...)
D=-19767512 m=-4941878    vptor=6   Cp=0.78474406738375920976115602
p=3, Interval [10^8, 10^8+10^7]
D=-100075971 m=-100075971  vptor=4   Cp=0.477101585455621088296257
D=-100080003 m=-100080003  vptor=5   Cp=0.596375677517090310811391
(...)
D=-100787315 m=-100787315  vptor=5   Cp=0.596147767754605081482216
D=-100867844 m=-25216961   vptor=6   Cp=0.715346318656502973478053
(...)
D=-117344127 m=-117344127  vptor=6   Cp=0.709521342553855083402930
D=-119846559 m=-119846559  vptor=7   Cp=0.826835890443221508331985
(...)
D=-135140024 m=-33785006   vptor=7   Cp=0.821531794828164970116186
D=-136159455 m=-136159455  vptor=8   Cp=0.938516745792290367614873
\end{verbatim}
\normalsize

\smallskip
For $D_K=-136159455$, we get
${\mathcal W}_K \simeq \Z/3\Z$, ${\mathcal T}_K \simeq \Z/3^7\Z \times \Z/3\Z$ 
and $\Cl_K \simeq \Z/3^6\Z \times \Z/3\Z$
(using the instruction ${\sf quadclassunit(-136159455)}$).

\subsubsection{Conclusion}
Consider a prime $p$ and a fixed degree $d$.
Assume, tentatively, a strong $\varepsilon$-conjecture for the groups 
${\mathcal T}_K$ in the case of totaly real number fields $K$ of degree $d$, 
which implies a strong $\varepsilon$-conjecture for the $p$-class groups.
We then have the existence, for all $\varepsilon>0$, of a constant 
$\wt C_{d,p,\varepsilon}$ such that:
$${\rm log}(\order{\mathcal T}_K) 
\leq {\rm log}(\wt C_{d,p,\varepsilon}) + \varepsilon \cdot {\rm log}\big (\sqrt {D_K}\,\big); $$

then introduce the function $\wt{\mathcal C}_{K,p}$:
\begin{equation}\label{ckp}
\wt{\mathcal C}_{K,p}  := \frac{{\rm log}(\order {\mathcal T}_K)}
{{\rm log}\big (\sqrt {D_K}\,\big)} 
 \leq  \frac{{\rm log}(\wt C_{d,p,\varepsilon})}{{\rm log}\big (\sqrt {D_K}\,\big)}  + \varepsilon. 
\end{equation}

So, when $D_K \to \infty$, we get $\wt{\mathcal C}_{K,p} = \varepsilon +o(1)$.
But in \cite[\S\,5.3]{Gr2}, we have proved that there exist explicit infinite families 
of real quadratic fields for which $\wt{\mathcal C}_{K,p} \approx 1$ (contradiction).

\smallskip
We obtain, generalizing to arbitrary degrees, the following heuristic:

\medskip
{\it There is no absolute strong $\varepsilon$-conjecture for the ${\mathcal T}_K$
groups of totally real number fields $K$ and ${\mathcal T}_K$ is essentially 
governed by the normalized $p$-adic regulator ${\mathcal R}_K$. Nevertheless,
as for $p$-class groups, one may conjecture that the exceptions to the
strong $\varepsilon$-conjecture are due to sparse subfamilies of density zero.}

\smallskip
Recall that, from \eqref{rank1} and \eqref{rank2}, the $p$-rank
$\varepsilon$-conjecture for the ${\mathcal T}_K$ does exist if and only 
if the $p$-rank $\varepsilon$-conjecture does exist for the 
$p$-class groups $\Cl_K$.

\subsection{Reflection theorem and $p$-rank $\varepsilon $-conjectures}\label{refl}
Another justification of the above comments is to recall the reflection theorem 
\cite{Gr7} which exchanges, roughly speaking, ``imaginary components'' of
$p$-class groups $\Cl$ with ``real components'' of
$p$-torsion groups ${\mathcal T}$, and conversely, subject to consider 
fields $K$ containing the group $\mu_p$ of $p$th roots of unity.

\smallskip
In full generality, the following result precises \eqref{rank1}
and \eqref{rank2} when $\mu_p\subset K$:

\begin{proposition}[{\cite[Theorem III.4.2.2]{Gr1}}] Let $K$ be a 
number field containing $\mu_p$ and fulfilling Leopoldt's conjecture at~$p$. 
Then we have the rank formula (reflection theorem)\,\footnote{The 
mentions ``ord'', ``res'' are related to the case $p=2$; for
$p > 2$, since $\mu_p \subset K$, the two notions coincide.}
${\rm rk}_p({\mathcal T}_K^{\rm ord}) = 
{\rm rk}_p(\Cl_K^{S\,{\rm res}}) + \order S -1$,  
where $S$ is the set of $p$-places of $K$ and $\Cl_K^{S\,{\rm res}}$ the
$S$-class group $\Cl_K^{\rm res}/\cl^{\rm res}(S)$.
\end{proposition}

So, ${\rm rk}_p(\Cl_K^{S\,{\rm res}}) = 
{\rm rk}_p({\mathcal T}_K^{\rm ord}) - (\order S -1)$ yields:
\begin{equation*}
\begin{aligned}
{\rm rk}_p(\Cl_K^{\rm res}) & \leq 
{\rm rk}_p(\Cl_K^{S\,{\rm res}}) + {\rm rk}_p(\cl^{\rm res}(S)) \\
& \leq {\rm rk}_p({\mathcal T}_K^{\rm ord}) - \big [ (\order S -1) 
- {\rm rk}_p(\cl^{\rm res}(S))\big ]
\leq {\rm rk}_p({\mathcal T}_K^{\rm ord}) +1 ,
\end{aligned}
\end{equation*}

giving :
$$-1 \leq {\rm rk}_p({\mathcal T}_K^{\rm ord}) - {\rm rk}_p(\Cl_K^{\rm res})
\leq \order S -1; $$

since $\order S - 1$ is bounded in the family of number fields, 
of fixed degree $d$, this relation shows that any $p$-rank
$\varepsilon $-conjecture, true for an invariant, is fulfilled by the other.
When $\mu_p \not\subset K$, one must use 
the field $K(\mu_p)$ and the general reflection theorem with 
$p$-adic characters (\cite[Theorem II.5.4.5]{Gr1}, \cite{Gr7}).

\medskip
The reflection theorem has been used in \cite{EV}, in a different 
manner, using many split primes in $K(\mu_p)$, the notion of Arakelov 
class group (see e.g., \cite{Sc3}), and the following analytic explanation
by the authors:

\smallskip
{\it Roughly, the point is that small non-inert primes in a number 
field represent elements of the class group which tend not to satisfy 
any relation with small coefficients. Thus the existence of
many such primes contributes significantly to the quotient of the 
class group by its $\ell$-torsion, yielding the desired upper bounds.}

\medskip
It would be interesting to deepen these approaches that have 
connections through class field theory, complex and $p$-adic 
analytic methods.

\section*{Acknowledgments} We thank Peter Koymans and 
Carlo Pagano for fruitful information about various approaches 
(density results  v.s. $\varepsilon$-conjectures) concerning 
$p$-class groups in degree $p$ cyclic fields.

\end{document}